\documentclass[11pt,twoside,reqno]{amsart}

\usepackage{mathtools}
\usepackage{microtype}
\usepackage{geometry}
\usepackage{xcolor}
\usepackage{hyperref}
\usepackage{fancyhdr}
\usepackage{amsmath}             
\usepackage{amssymb}           
\usepackage{amsfonts}           
\usepackage{latexsym}           
\usepackage{amsthm}              
\usepackage{bbm}
\usepackage{cite}
\usepackage{tikz}
\usepackage{comment}
\usetikzlibrary{decorations.pathreplacing}
\usetikzlibrary{angles}
\usetikzlibrary{calc}
\usepackage[justification=centering]{caption}
\usepackage{type1cm}
\usepackage{lmodern}

\hypersetup{
  colorlinks=true,
  linkcolor=black,
  anchorcolor=black,
  citecolor=black,
  filecolor=black,      
  menucolor=red,
  runcolor=black,
  urlcolor=black,
}

\numberwithin{equation}{section}

\theoremstyle{plain}
\newtheorem{theorem}{Theorem}[section]

\newtheorem{proposition}[theorem]{Proposition}

\newtheorem{lemma}[theorem]{Lemma}

\theoremstyle{remark}
\newtheorem{remark}[theorem]{Remark}

\newtheorem{example}[theorem]{Example}

\theoremstyle{definition}

\newtheorem*{question*}{Question}

\newcommand{\LL}{\mathcal{L}}

\newcommand{\CC}{\mathcal{C}}

\newcommand{\R}{\mathbb{R}}
\newcommand{\RP}{\mathbb{RP}^1}

\newcommand{\Z}{\mathbb{Z}}
\newcommand{\N}{\mathbb{N}}

\newcommand{\iii}{\mathtt{i}}
\newcommand{\jjj}{\mathtt{j}}

\newcommand{\fii}{\varphi}

\newcommand{\A}{\mathsf{A}}

\newcommand{\Xl}{T_{\lambda}}

\newcommand{\projspan}[1]{\langle#1\rangle}

\renewcommand{\ge}{\geqslant}
\renewcommand{\le}{\leqslant}
\renewcommand{\geq}{\geqslant}
\renewcommand{\leq}{\leqslant}

\DeclareMathOperator{\Bad}{Bad}

\DeclareMathOperator{\dimloc}{dim_{loc}}
\DeclareMathOperator{\udimloc}{\overline{dim}_{loc}}
\DeclareMathOperator{\ldimloc}{\underline{dim}_{loc}}

\DeclareMathOperator{\dimm}{dim_M}
\DeclareMathOperator{\udimm}{\overline{dim}_M}
\DeclareMathOperator{\ldimm}{\underline{dim}_M}

\DeclareMathOperator{\dimh}{dim_H}

\DeclareMathOperator{\diml}{dim_L}
\DeclareMathOperator{\dimaff}{dim_{aff}}

\DeclareMathOperator{\dima}{dim_A}

\DeclareMathOperator{\Tan}{Tan}

\DeclareMathOperator{\dist}{dist}
\DeclareMathOperator{\diam}{diam}

\DeclareMathOperator{\proj}{proj}

\DeclareMathOperator{\rank}{rank}

\DeclareMathOperator{\im}{im}
\DeclareMathOperator{\id}{Id}
\renewcommand{\atop}[2]{\genfrac{}{}{0pt}{}{#1}{#2}}

\def\MarkRightAngle[size=#1](#2,#3,#4){
  \draw[black] ($(#3)!#1!(#2)$) -- ($($(#3)!#1!(#2)$)!#1!90:(#2)$) -- ($(#3)!#1!(#4)$)
}

\begin{document}

\title{Slices of the Takagi function}

\author{Roope Anttila}
\address[Roope Anttila]
        {Research Unit of Mathematical Sciences \\ 
         P.O.\ Box 8000 \\ 
         FI-90014 University of Oulu \\ 
         Finland}
\email{roope.anttila@oulu.fi}

\author{Bal\'azs B\'ar\'any}
\address[Bal\'azs B\'ar\'any]
        {Department of Stochastics \\
        	Institute of Mathematics \\
        	Budapest University of Technology and Economics \\
        	M\H{u}egyetem rkp. 3 \\
        	H-1111 Budapest,
        	Hungary}
\email{balubsheep@gmail.com}

\author{Antti K\"aenm\"aki}
\address[Antti K\"aenm\"aki]
        {Research Unit of Mathematical Sciences \\ 
         P.O.\ Box 8000 \\ 
         FI-90014 University of Oulu \\ 
         Finland}
\email{antti.kaenmaki@oulu.fi}

\thanks{R. Anttila was financially supported by the Magnus Ehrnrooth foundation. B. B\'ar\'any was financially supported by the grants NKFI FK134251, K142169, and the grant NKFI KKP144059 ``Fractal geometry and applications''.}
\date{\today}

\begin{abstract}
  We show that the Hausdorff dimension of any slice of the graph of the Takagi function is bounded above by the Assouad dimension of the graph minus one, and that the bound is sharp. The result is deduced from a statement on more general self-affine sets, which is of independent interest. We also prove that the Marstrand's slicing theorem on the graph of the Takagi function extends to all slices if and only if the upper pointwise dimension of every projection of the length measure on the $x$-axis lifted to the graph is at least one.
\end{abstract}

\maketitle

\section{Introduction}

The \emph{Takagi function} $T_{\lambda} \colon [0,1] \to \R$ for the parameter $\tfrac12 < \lambda < 1$ is defined by setting
\begin{equation} \label{eq:takagi-def}
  T_{\lambda}(x) = \sum_{n=0}^\infty \lambda^n \dist(2^nx,\Z)
\end{equation}
for all $x \in [0,1]$. In mathematical writing it is customary to distinguish a function from its graph. Notwithstanding, we stick to the definition of a function as a total and univalent binary relation which in our case is convenient notation-wise as then $T_\lambda$ denotes both the function and its graph. The Takagi function, being continuous yet having at no point a finite derivative, is one of the famous examples of ``pathological functions''. For the basic properties of the Takagi function and a summary of recent research the reader is referred to the surveys of Allaart and Kawamura \cite{AllaartKawamura2012} and Lagarias \cite{Lagarias2012}.

Level sets of the Takagi function, i.e.\ the sets of points $x \in \R$ at which $T_\lambda(x)$ equals a given value, have been studied extensively; see \cite{LagariasMaddock2012,LiuLi2015,Allaart2014,Allaart2013,Buczolich2008,Allaart2012,Allaart2012b,LagariasMaddock2011,AmoBhouriDiazFernandez2011}. Such level sets appear as horizontal slices of $T_\lambda$ meaning that they are intersections $T_\lambda \cap (V_0+x)$, where $V_0$ is the $x$-axis and $x \in \R^2$. When $\lambda = \tfrac12$, it has been proven that the Hausdorff dimension of slices with integer slope is at most $\tfrac12$, and the bound is attained by some slice; see \cite{AmoDiazFernandez2013,Maddock2010}. In this paper, we obtain a sharp bound for the Hausdorff dimensions of all slices of $T_\lambda$, when $\tfrac12 < \lambda < 1$, in terms of the Assouad dimension of $T_\lambda$.

The study of the dimensions of slices has a rich history. The classical Marstrand's slicing theorem \cite{Marstrand1954} shows that almost every fiber of a projection does not store more dimension than what is the surplus. We denote the Hausdorff dimension by $\dimh$ and the collection of all lines in $\R^2$ passing through the origin by $\RP$. The slicing theorem states that, given a Borel set $X \subset \R^2$ and $V \in \RP$, we have
\begin{equation} \label{eq:marstrand-slicing}
  \dimh(X \cap (V+x)) \le \max\{0,\dimh(X) - 1\}
\end{equation}
for Lebesgue almost all $x \in V^\bot$. Often when the set $X$ has some additional arithmetic or geometric structure, stronger statements can be made about dimensions of all slices. For example, if $X=A\times B$ where $A$ and $B$ are invariant under the maps $x\mapsto 2x\mod1$ and $x\mapsto3x\mod1$, then the bound in \eqref{eq:marstrand-slicing} holds for all slices, except those in the directions of the coordinate axes. This celebrated result was first conjectured by Furstenberg \cite{Furstenberg1970} and recently proved independently and simultaneously by Shmerkin \cite{Shmerkin2019} and Wu \cite{Wu2019}.

The Takagi function is an example of a self-affine set. For many sets in this class, the Hausdorff dimensions of slices are closely connected to the Assouad dimensions of the sets; see Section \ref{sec:preliminaries} for the relevant definitions. This was first observed by Mackay \cite{Mackay2011}, who expressed the Assouad dimensions of a special class of self-affine sets called Bedford-McMullen carpets in terms of the dimensions of their projections on the $x$-axis and the dimensions of their slices in the direction of the $y$-axis. Algom \cite{Algom} showed that the Minkowski dimension of any slice, which is not in the direction of the coordinate axes, of certain Bedford-McMullen carpets $X$, is bounded above by $\max\{0,\dima(X)-1\}$. Here $\dima(X)$ denotes the Assouad dimension of $X$ and is always bounded from below by the Hausdorff dimension. Recently, B\'ar\'any, K\"aenm\"aki, and Yu \cite[Theorem 1.3]{BaranyKaenmakiYu2021-preprint} showed that a similar phenomenon is also present for certain totally disconnected self-affine sets. In fact, in the class of self-affine sets they consider, the upper bound $\max\{0,\dima(X) - 1\}$ is achieved and there are examples of self-affine sets in this class for which $\dimh(X) < \dima(X)$. However, since the Takagi function is connected, the results in \cite{BaranyKaenmakiYu2021-preprint} do not apply.

Utilizing the self-affinity of the Takagi function, B\'ar\'any, Hochman, and Rapaport \cite[Corollary 7.6]{BaranyHochmanRapaport2019} proved that
\begin{equation} \label{eq:takagi-hausdorff-dim}
  \dimh(T_\lambda) = 2 - \frac{\log\lambda}{\log \frac12} < 2;
\end{equation}
see also Ledrappier \cite{Ledrappier1992}. The Assouad dimension of the Takagi function was studied by Yu \cite{Yu2020} in some special cases. He showed that, if $2$ in the definition of the Takagi function \eqref{eq:takagi-def} is replaced by an integer greater than $3$, there exist parameters for which the Assouad dimension is strictly larger than the Hausdorff dimension. In the online version of the paper, he also conjectured that $\dima(T_\lambda)=2$ for all $\tfrac12 < \lambda < 1$. The following theorem is the first main result of the paper.

\begin{theorem} \label{thm:main1}
  If $T_\lambda$ is the Takagi function, then
  \begin{equation*}
      \max_{\atop{x\in T_\lambda}{V\in\RP}}\dimh(T_\lambda \cap (V+x)) = \dima(T_\lambda)-1 < 1
  \end{equation*}
\end{theorem}
This theorem is based on a result for a more general class of self-affine sets, Theorem \ref{thm:dimasosc}, which generalizes the results in \cite[\S 5]{BaranyKaenmakiRossi2021} and \cite[\S 5]{BaranyKaenmakiYu2021-preprint}.

As our second main result, we investigate when the bound \eqref{eq:marstrand-slicing} of the Marstrand's slicing theorem can be extended to all slices of the Takagi function. By Theorem \ref{thm:main1}, this happens precisely when $\dima(T_\lambda)=\dimh(T_\lambda)$. For a given $t \in \R$, let $\proj_t \colon \R^2 \to \R$, $\proj_t(x_1,x_2) = (x_1,x_2) \cdot (t,1)$. The pushforward of a measure $\mu$ is denoted by $f_*\mu$ whenever $f$ is a measurable mapping. We let $\nu = (\id,T_\lambda)_*\LL^1$ be the Lebesgue measure $\LL^1$ lifted to the Takagi function $T_\lambda$ and $\ldimloc(\mu,x)$ be the lower pointwise dimension of a measure $\mu$ at $x$. It follows from \eqref{eq:marstrand-slicing} and \eqref{eq:takagi-hausdorff-dim}, that
\begin{equation*}
  \ldimloc({\proj_t}_*\nu,\proj_t(x)) \ge 1
\end{equation*}
for $\nu$-almost all $x \in T_\lambda$ and Lebesgue almost all $t \in \R$. The following theorem is the second main result of the paper. We show that if the above lower bound holds for all $x$ and $t$, then the Marstrand's slicing theorem \eqref{eq:marstrand-slicing} is extended to all slices.

\begin{theorem} \label{thm:main2}
  If $T_\lambda$ is the Takagi function and $\nu = (\id,T_\lambda)_*\LL^1$ is the Lebesgue measure $\LL^1$ lifted to the Takagi function $T_\lambda$, then
  \begin{equation*}
      \max_{\atop{x\in T_\lambda}{V\in\RP}}\dimh(T_\lambda \cap (V+x))= 1 - \frac{\log\lambda}{\log \frac12}
  \end{equation*}
  if and only if
  \begin{equation*}
      \dimloc({\proj_t}_*\nu,\proj_t(x))\geq 1
  \end{equation*}
  for all $x \in T_\lambda$ and $t \in \R$, where $\dimloc$ denotes either the lower or the upper pointwise dimension.
\end{theorem}

We remark that
\begin{equation*}
  \ldimloc({\proj_t}_*\nu,\proj_t(x))\geq 1
\end{equation*}
holds for all $x \in T_\lambda$ and $t \in \R$ if and only if the $L^q$-dimension of the measure ${\proj_t}_*\nu$ equals to one for all $t \in \R$ and $q>0$, i.e.,
\begin{equation*}
  \inf_{t \in \R}\liminf_{r \downarrow 0}\min_{x\in T_{\lambda}}\frac{\log({\proj_t}_*\nu(B(\proj_t(x),r))}{\log r}=1.
\end{equation*}
We refer to the definition and basic properties of the $L^q$-dimension for \cite[\S 1.3]{Shmerkin2019survey}, and leave the proof of the above fact as an exercise to the interested reader.

The rest of the paper is organized as follows. In Section \ref{sec:preliminaries}, we recall some basic results in dimension theory and establish the general setting of self-affine sets we will be working with. The results in the general setting are presented in Sections \ref{sec:tangent-decompo} and \ref{sec:self-affine-large}. We will then specialize to the Takagi function and prove Theorem \ref{thm:main1} in Section \ref{sec:proof-main1} and Theorem \ref{thm:main2} in Section \ref{sec:proof-main2}.

\section{Notation and preliminaries}\label{sec:preliminaries}

\subsection{Dimensions and weak tangents}
Let us briefly recall definitions of some of the basic notions of dimension used in fractal geometry. The \emph{Hausdorff dimension} of a set $X\subset \R^2$ is
\begin{align*}
    \dimh(X)= \inf \{s>0\colon &\text{for every } \varepsilon>0 \text{ there is } \{U_i\}_{i\in\N} \text{ such} \\ &\text{that } X\subset \bigcup_{i\in\N}U_i\text{ and }\sum_{i\in\N}\diam(U_i)^s<\varepsilon\}.
\end{align*}
The \emph{lower and upper pointwise dimensions} of a Borel measure $\mu$ at $x \in \R^2$ are
\begin{align*}
\ldimloc(\mu,x) &= \limsup_{r \downarrow 0} \frac{\log \mu(B(x,r))}{\log r}, \\
\udimloc(\mu,x) &= \liminf_{r \downarrow 0} \frac{\log \mu(B(x,r))}{\log r},
\end{align*}
respectively. We assume familiarity with the basic properties of the Hausdorff dimension and pointwise dimensions, and how they are connected; see for example \cite{Mattila1995,Falconer1997}. If $X$ is bounded, then the \emph{$r$-covering number} of $X$,
\begin{equation*}
  N_r(X) = \min\{ k \in \N \colon X \subset \bigcup_{i=1}^k B(x_i,r) \text{ for some } x_1,\ldots,x_k \in \R^2 \},
\end{equation*}
is the least number of closed balls of radius $r>0$ needed to cover $X$. The \emph{lower and upper Minkowski dimensions} of a bounded set $X\subset \R^2$ are
\begin{align*}
    \ldimm(X)&=\liminf_{r \downarrow 0}\frac{\log N_{r}(X)}{-\log r}, \\
    \udimm(X)&=\limsup_{r\downarrow 0}\frac{\log N_{r}(X)}{-\log r},
\end{align*}
respectively. In the case that the limit exists, it is denoted by $\dimm(X)$ and called the \emph{Minkowski dimension} of $X$. The \emph{Assouad dimension} of $X \subset \R^2$ is
\begin{align*}
    \dima(X)=\inf\{s>0\colon &\text{there exists } C>0 \text{ such that} \\ 
    &\text{for every }x\in X \text{ and } 0<r<R \\
    &\text{it holds that } N_{r}(X\cap B(x,R))\leq C(\tfrac{R}{r})^s\}.
\end{align*}
The Assouad dimension is designed to capture the extremal scaling behaviour of the set by quantifying the size of the least doubling parts of the set in question. The basic inequality we will use repeatedly is
\begin{equation*}
    \dimh(X)\leq\ldimm(X)\leq \udimm(X)\leq\dima(X)
\end{equation*}
for all bounded sets $X\subset \R^2$. For the proof of this and other basic properties of the Assouad dimension, we refer to \cite{Fraser14}.

The concept of weak tangents has proven to be very useful in the study of the Assouad dimension. Let $X$ be a compact subset of $\R^2$. For $x\in X$ and $r>0$ we denote by $M_{x,r}\colon \R^2\to\R^2$ the linear map
\begin{equation*}
    M_{x,r}(y)=\frac{y-x}{r}.
\end{equation*}
Note that $M_{x,r}(B(x,r))=B(0,1)$. A set $T$ which intersects the interior of $B(0,1)$ is called a \emph{weak tangent} of $X$ if there are a sequence $(x_n)_{n\in\N}$ of points in $X$ and a sequence $(r_n)_{n\in\N}$ of positive real numbers converging to $0$ such that
\begin{equation*}
    M_{x_n,r_n}(X)\cap B(0,1)\to T,
\end{equation*}
in Hausdorff distance. The collection of all weak tangents of $X$ is denoted by $\Tan(X)$. It is easy to see that a dimension of a weak tangent is a lower bound for the Assouad dimension of $X \subset \R^2$, i.e.\ $\dima(X)\geq \dima(T)$ for all $T\in\Tan(X)$; see e.g.\ \cite[Theorem 5.1.2]{Fraser2020}. Käenmäki, Ojala, and Rossi \cite[Proposition 5.7]{KOR} proved the following stronger result, which shows that the Assouad dimension of a compact set is realized by the maximal Hausdorff dimension of its weak tangents.

\begin{lemma}\label{lemma:kaenmakiojalarossi}
  If $X\subset \R^2$ is compact, then $\dima(X)=\max\{\dimh(T)\colon T\in\Tan(X)\}$.
\end{lemma}

The result introduces a way to obtain an upper bound for the Assouad dimension by bounding the Hausdorff dimension of every weak tangent.

\subsection{Real projective line and matrices}
Define an equivalence relation $\sim$ on $\R^2\setminus\{0\}$ by setting $v\sim w$ if and only if $v=cw$ for some $c\in\R$. Denote the equivalence class of $v \in \R^2\setminus \{0\}$ under this relation by $\projspan{ v}$. An elementary observation is that for any $0\ne c\in\R$ and $v\in \R^2\setminus\{0\}$ we have $\projspan{ cv} =\projspan{ v}$. Geometrically, $\projspan{ v}= \{w\in\R^2\colon w=cv \text{ and }c\in\R\}\subset \R^2$ is a line in $\R^2$ in the direction of $v$ passing through the origin. The \emph{real projective line} is $\RP= \{\projspan{ v}\colon v\in\R^2\setminus\{0\}\}$. An element of $\RP$ is called a \emph{line}. If the representative of an element of $\RP$ is left implicit, we use capital letters such as $V$ or $W$ to refer to the element. We let $\sphericalangle\colon \RP\to\R$ denote the metric on $\RP$ given by
\begin{equation*}
    \sphericalangle(\projspan{v},\projspan{w})=\arccos\left(\frac{|v\cdot w|}{\| v\|\| w\|}\right)=\arcsin\left(\frac{\|v\wedge w\|}{\|v\|\|w\|}\right),
\end{equation*}
where $v\cdot w$ and $v\wedge w$ denote the inner product and exterior product of the vectors $v$ and $w$, respectively. In other words, the distance between two lines is given by the smaller of the angles between them. A ball in this metric is called a \emph{projective interval}. With the topology induced by the metric, the map $v\mapsto \projspan{ v}$ from $\R^2\setminus\{0\}$ to $\RP$ is continuous.

The group of invertible $2\times 2$ matrices is denoted by $GL_2(\R)$. A matrix $A\in GL_2(\R)$ induces an action on $\RP$ by
\begin{equation*}
    A\projspan{ v} = \projspan{ Av}.
\end{equation*}
For any $V\in\RP$, we denote by $\proj_{V}\colon \R^2\to V$ the \emph{orthogonal projection} onto the subspace $V$, that is, $\proj_V$ is the unique linear map satisfying $\proj_V|_V=\mathrm{Id}|_V$ and $\ker(\proj_V)=V^{\perp}$. It is easy to see, consult e.g.\ \cite[Lemma 2.1]{KaenmakiNissinen-preprint}, that a rank one $2\times 2$ matrix $A$ is bi-Lipschitz equivalent to $\proj_{\ker(A)^{\perp}}$.

The \emph{singular values} $\alpha_1(A)$ and $\alpha_2(A)$ of a matrix $A\in GL_2(\R)$ are the square roots of the non-negative eigenvalues of the positive definite matrix $A^{\top}A$, ordered so that $\alpha_1(A)\geq \alpha_2(A)$. Note that $\alpha_1(A)$ and $\alpha_2(A)$ are the lengths of the semiaxes of the ellipse $A(B(0,1))$. If $A \in GL_2(\R)$ is such that $\alpha_1(A)>\alpha_2(A)$, then we let $\eta_1(A)$ be one of the two unit eigenvectors of $A^{\top}A$ corresponding to the eigenvalue $\alpha_1(A)^2$. If $\alpha_1(A)=\alpha_2(A)$, then we write $\eta_1(A)=S^1=\{x\in\R^2\colon|x|=1\}$. Observe that $\alpha_1(A)=\|A\|=\|A|\langle\eta_1(A)\rangle\|$, $\alpha_2(A)=\|A^{-1}\|^{-1}=\|A^{-1}|\langle\eta_1(A^{-1})\rangle\|^{-1}$, and $\alpha_1(A)\alpha_2(A)=|\det(A)|$.

\subsection{Self-affine set and shift space}
An \emph{iterated function system (IFS)} is a finite tuple of contractive maps $\Phi=(\varphi_1,\ldots,\fii_N)$ acting on $\R^2$. By a classical result of Hutchinson \cite{Hutchinson1981}, $\Phi$ admits a unique non-empty compact set, denoted by $X$, satisfying
\begin{equation*}
    X=\bigcup_{i=1}^N\varphi_i(X).
\end{equation*}
We call $X$ the \emph{limit set} of $\Phi$. We say that $\Phi$ is an \emph{affine IFS} if the maps $\varphi_i$ are affine, i.e.\ $\varphi_i(x)=A_ix+b_i$, where $A_i\in GL_2(\R)$ and $b_i\in\R^2$. In this case, the corresponding limit set is called a \emph{self-affine set}. We use the convention that whenever we speak about a self-affine set $X$, then it is automatically accompanied with a tuple of affine maps which defines it. A self-affine set is said to satisfy the \emph{strong separation condition (SSC)} if $\varphi_i(X)\cap \varphi_j(X) = \emptyset$ for all $i\ne j$, and the \emph{strong open set condition (SOSC)} if there exists an open set $U$ such that $X\cap U\ne\emptyset$, $\varphi_i(U)\subset U$ for all $i\in\{1,\ldots,N\}$, and $\varphi_i(U)\cap \varphi_j(U) = \emptyset$ whenever $i\ne j$.

Given an IFS, we consider the symbolic representation of the limit set $X$ as follows. Let $\Sigma =\{1,\ldots,N\}^{\N}$ denote the collection of all \emph{infinite words} obtained by concatenating digits in $\{1,\ldots,N\}$. Similarly, $\Sigma_n=\{1,\ldots,N\}^n$ is the set of \emph{finite words of length} $n \in \N$, and $\Sigma_*=\bigcup_{n\in\N}\Sigma_n$ is the set of finite words of any length. Given $\iii=i_1i_2\cdots\in\Sigma$, we define $\iii|_n=i_1\cdots i_n$ to be the restriction of $\iii$ to its first $n$ indices, and given $\iii=i_1\cdots i_n\in\Sigma_n$, let $\iii^-=\iii|_{n-1}=i_1\cdots i_{n-1}\in\Sigma_{n-1}$ and $\overleftarrow{\iii}=i_n\cdots i_1$ be the word obtained from $\iii$ by reversing the order of its digits. The concatenation of two words $\iii\in\Sigma_*$ and $\jjj\in\Sigma_*\cup\Sigma$ is denoted by $\iii\jjj$. Given $\iii\in\Sigma_*$, the infinite word obtained by concatenating $\iii$ with itself infinitely many times is denoted by $\overline{\iii}$, that is, $\overline{\iii}=\iii\iii\cdots$. For two finite or infinite words $\iii$ and $\jjj$, their longest common prefix is denoted by $\iii\wedge\jjj$, and the length of a word $\iii$ is denoted by $|\iii|$. We define $\sigma\colon \Sigma\to\Sigma$ by setting $\sigma\iii=\sigma(\iii)=i_2i_3\cdots$ for all $\iii=i_1i_2\cdots\in\Sigma$, and call it the \emph{left shift}. Given $n\in\N$ and $\iii\in\Sigma_n$, we define the \emph{cylinder set} by $[\iii]=\{\jjj\in\Sigma\colon \jjj|_n=\iii\}$. The \emph{shift space} $\Sigma$ is a compact topological space in the topology whose base is the collection of all cylinder sets. Alternatively, a metric $\varrho$ on $\Sigma$ defined by
\begin{equation*}
    \varrho(\iii,\jjj)=2^{-|\iii\wedge\jjj|},
\end{equation*}
with the interpretation that $2^{-\infty}=0$, induces the same topology as the open balls in this metric are precisely the cylinder sets. It is also worth pointing out that the cylinder sets are open and closed in this topology and generate the Borel $\sigma$-algebra. A map $f\colon \Sigma\to M$, where $(M,d)$ is a metric space, is \emph{H\"older continuous}, if there are constants $C,\alpha>0$ such that
\begin{equation*}
    d(f(\iii),f(\jjj))\leq C\alpha^{|\iii\wedge\jjj|},
\end{equation*}
for all $\iii,\jjj\in\Sigma$. Finally, for a given IFS $(\fii_1,\ldots,\fii_N)$ and its limit set $X$, we define the \emph{canonical projection} $\pi \colon \Sigma \to X$ by setting
\begin{equation*}
  \pi\iii = \pi(\iii) = \lim_{n \to \infty} \fii_{i_1} \circ \cdots \circ \fii_{i_n}(\bar{0})
\end{equation*}
for all $\iii = i_1i_2\cdots \in \Sigma$, where $\bar{0}=(0,0)$. It is evident that $\pi$ is H\"older continuous.

\subsection{Semigroup and domination}
Understanding the semigroup generated by $\A = (A_1,\ldots,A_N)\in GL_2(\R)^N$ is crucial in the study of self-affine sets. In this context, it is rather standard practise to use $\Sigma_*$ to index the elements in the semigroup. Indeed, we write
\begin{equation*}
  A_{\iii}=A_{i_1}\cdots A_{i_n}
\end{equation*}
for all $\iii=i_1\cdots i_n \in \Sigma_n$ and $n \in \N$. Our standing assumption is that $\A$ is \emph{dominated}, that is, there exist constants $C>0$ and $0<\tau<1$ such that
\begin{equation} \label{eq:domination-def}
    \alpha_2(A_{\iii})\leq C\tau^{|\iii|}\alpha_1(A_{\iii})
\end{equation}
for all $\iii\in\Sigma_*$. Domination ensures that when iteratively applying the matrices in $\A$ to the unit ball, the resulting ellipses get thinner and thinner at an exponential rate. We say that a self-affine set $X$ is dominated if the tuple consisting of the linear parts of the maps in the associated affine IFS is. A proper subset $\CC\subset \RP$ is called a \emph{multicone} if it is a finite union of closed projective intervals. A multicone $\CC\subset \RP$ is \emph{strongly invariant} for $\A$ if $A_i\CC\subset \CC^{\circ}$ for all $i\in\{1,\ldots,N\}$, where $\CC^{\circ}$ denotes the interior of $\CC$. By \cite[Theorem B]{BochiGourmelon2009}, $\A$ admits a strongly invariant multicone if and only if $\A$ is dominated. It is a simple fact that if $\CC$ is a strongly invariant multicone for $\A$, then $\overline{\RP\setminus \CC}$ is a strongly invariant multicone for $\A^{-1}= (A_1^{-1},\ldots,A_{N}^{-1})$. Write
\begin{align*}
    A_{\overleftarrow{\iii}}^{\top}&=(A_{\overleftarrow{\iii}})^{\top}=A_{i_1}^{\top}\cdots A_{i_n}^{\top},\\
    A_{\overleftarrow{\iii}}^{-1}&=(A_{\overleftarrow{\iii}})^{-1}=A_{i_1}^{-1}\cdots A_{i_n}^{-1},
\end{align*}
and let
\begin{align*}
    \vartheta_1(\iii)&=\projspan{ A_{\iii}\eta_1(A_{\iii})},\\
    \vartheta_2(\iii)&=\projspan{ A_{\overleftarrow{\iii}}^{-1}\eta_1(A_{\overleftarrow{\iii}}^{-1})},
\end{align*}
for all $\iii\in\Sigma_n$ and $n\in\N$. The geometric interpretation is that $\vartheta_1(\iii)$ and $\vartheta_2(\iii)$ correspond to the orientation of the principal semiaxis of the ellipses $A_{\iii}(B(0,1))$ and $A_{\overleftarrow{\iii}}^{-1}(B(0,1))$, respectively. We also define
\begin{equation*}
  \overline{\vartheta}_k(\iii) = \lim_{n\to\infty}\vartheta_k(\iii|_n)
\end{equation*}
for all $\iii \in \Sigma$ and $k \in \{1,2\}$ whenever the limit exists. The following lemma guarantees that under domination, the limit exists at every point and therefore, we have defined a map $\overline{\vartheta}_k \colon \Sigma \to \RP$.

\begin{lemma}\label{lemma:rossi}
  If $\A=(A_1,\ldots,A_N) \in GL_2(\R)^N$ is dominated and $\CC \subset \RP$ is a strongly invariant multicone for $\A$, then, for $k\in\{1,2\}$,
  \begin{enumerate}
    \item\label{it:rossi1} the limit $\overline{\vartheta}_k(\iii)=\lim_{n\to\infty}\vartheta_k(\iii|_n)$ exists for all $\iii\in\Sigma$ and the convergence is uniform,
    \item\label{it:rossi2} the map $\overline{\vartheta}_k\colon \Sigma\to \RP$ is Hölder continuous,
    \item\label{it:rossi3} the set $\overline{\vartheta}_k(\Sigma)$ is compact and contains the accumulation points of $\{\vartheta_k(\iii)\colon \iii\in\Sigma_*\}$,
    \item\label{it:rossi4} $A_{\iii}\overline{\vartheta}_1(\jjj)=\overline{\vartheta}_1(\iii\jjj)$ and $A_{\overleftarrow{\iii}}^{-1}\overline{\vartheta}_2(\jjj)=\overline{\vartheta}_2(\iii\jjj)$ for all $\iii\in\Sigma_*$ and $\jjj\in\Sigma$,
    \item\label{it:rossi5} $\overline{\vartheta}_1(\Sigma)\subset \CC^{\circ}$ and $\overline{\vartheta}_2(\Sigma)\subset \RP\setminus \CC$.
  \end{enumerate}
\end{lemma}

\begin{proof}
For $k=1$, the claims \eqref{it:rossi1}, \eqref{it:rossi3}, and \eqref{it:rossi4} are proved in \cite[Lemma 2.1]{Rossi2021} and \eqref{it:rossi5} follows from the definition of the strongly invariant multicone. One can repeat the proofs for the dominated tuple $\A^{-1}=(A_1^{-1},\ldots,A_2^{-1})$ to obtain the claims for $k=2$. Similarly, it suffices to prove \eqref{it:rossi2} for $k=1$.

To that end, let $\iii\in\Sigma$, $m\in\N$, and
\begin{equation*}
    \theta_m=\sphericalangle(\vartheta_1(\iii|_m),\vartheta_1(\iii|_{m+1})).
\end{equation*}
In the proof of \cite[Lemma 2.1]{KaenmakiKoivusaloRossi2017}, it was shown that there is $c>1$ not depending on $m$ such that
\begin{equation*}
    \sin(\theta_m)\leq c\frac{\alpha_2(A_{\iii|_m})}{\alpha_1(A_{\iii|_m})}.
\end{equation*}
Since $\vartheta_1(\iii|_m) \to \overline{\vartheta}_1(\iii)$ as $m \to \infty$ there exists $n_0 \in \N$ such that for every $m\ge n_0$ we have $\theta_m \le 2\sin(\theta_m)$ and, by recalling the definition of domination from \eqref{eq:domination-def},
\begin{align*}
    \sphericalangle(\vartheta_1(\iii|_n),\overline{\vartheta}_1(\iii))&\leq \sum_{m=n}^{\infty}\theta_m\leq 2c\sum_{m=n}^{\infty}\frac{\alpha_2(A_{\iii|_m})}{\alpha_1(A_{\iii|_m})}\leq 2cC\sum_{m=n}^{\infty}\tau^m=\frac{2cC}{1-\tau}\tau^n
\end{align*}
for all $n \ge n_0$. For every $\iii,\jjj\in\Sigma$ with $n = |\iii \wedge \jjj| \ge n_0$, we thus have
\begin{equation*}
    \sphericalangle(\overline{\vartheta}_1(\iii),\overline{\vartheta}_1(\jjj))\leq\sphericalangle(\overline{\vartheta}_1(\iii),\vartheta_1(\iii|_n))+ \sphericalangle(\vartheta_1(\iii|_n),\overline{\vartheta}_1(\jjj))\leq \frac{4cC}{1-\tau}\tau^{|\iii\wedge\jjj|}
\end{equation*}
and the map $\overline{\vartheta}_1 \colon \Sigma \to \RP$ is Hölder continuous.
\end{proof}

For a dominated matrix tuple $\A = (A_1,\ldots,A_N) \in GL_2(\R)^N$, the sets
\begin{align*}
  Y_F &= \{\im(A)\in\RP\colon A\in\overline{\{cA_{\iii}\colon c\in\R\text{ and }\iii\in\Sigma_*\}} \text{ has rank one}\}, \\
  X_F &= \{\im(A)\in\RP\colon A\in\overline{\{cA_{\overleftarrow{\iii}}^{-1}\colon c\in\R\text{ and }\iii\in\Sigma_*\}} \text{ has rank one}\},
\end{align*}
are the collections of \emph{forward and backward Furstenberg directions}, respectively. The following lemma gives useful characterizations for the sets $X_F$ and $Y_F$.

\begin{lemma}\label{lemma:alternate-furstenberg-dir}
  If $\A=(A_1,\ldots,A_N)\in GL_2(\R)^N$ is dominated and $\CC \subset \RP$ is a strongly invariant multicone for $\A$, then
  \begin{equation*}
      Y_F=\overline{\vartheta}_1(\Sigma)=\bigcap_{n=1}^{\infty}\bigcup_{\iii\in\Sigma_n}A_{\iii}\CC
      \qquad\text{and}\qquad
      X_F=\overline{\vartheta}_2(\Sigma)=\bigcap_{n=1}^{\infty}\bigcup_{\iii\in\Sigma_n}A^{-1}_{\overleftarrow{\iii}}\overline{\RP\setminus\CC}.
  \end{equation*}
\end{lemma}

\begin{proof}
We prove the claims for $Y_F$ and note that the claims for $X_F$ follow similarly by considering the dominated tuple $\A^{-1}=(A_1^{-1},\ldots,A_N^{-1})$. Let us first show that $Y_F \subset \overline{\vartheta}_1(\Sigma)$. To that end, let $V\in Y_F$ and choose a sequence $(\iii_n)_{n\in\N}$ of finite words and a sequence $(c_n)_{n\in\N}$ of real numbers such that $c_nA_{\iii_n}\to A$ and $\im(A)=V$. By passing to a sub-sequence if necessary, we may assume that
\begin{equation*}
    \eta_1(A_{\iii_n})\to \eta
\end{equation*}
for some $\eta\in S^1$. Since the maps $A_{\iii_n}$ are linear and $\sup_{n\in\N}\|c_nA_{\iii_n}\|<\infty$, it follows from the Banach-Steinhaus theorem that $c_nA_{\iii_n}\eta_1(A_{\iii_n})\to A\eta$ and therefore,
\begin{equation*}
    \|A\eta\|=\lim_{n\to\infty}c_n\|A_{\iii_n}\eta_1(A_{\iii_n})\|=\lim_{n\to\infty}c_n\|A_{\iii_n}\|=\|A\|.
\end{equation*}
In particular $\|A\eta\|>0$, so $A\eta$ is a non-zero vector in $\im(A)$. Thus, by the continuity of the map $v\mapsto \projspan{ v}$,
\begin{equation*}
    V=\im(A)=\projspan{ A\eta}=\lim_{n\to\infty}\projspan{ c_nA_{\iii_n}\eta(A_{\iii_n})}=\lim_{n\to\infty}\vartheta_1(\iii_n),
\end{equation*}
and $V\in \overline{\vartheta}_1(\Sigma)$ by Lemma \ref{lemma:rossi}\eqref{it:rossi3}.

Let us then show that $\overline{\vartheta}_1(\Sigma) \subset \bigcap_{n=1}^{\infty}\bigcup_{\iii\in\Sigma_n}A_{\iii}\CC$. Fix $V \in \overline{\vartheta}_1(\Sigma)$ and let $\iii \in \Sigma$ be such that $\overline{\vartheta}_1(\iii)=V$. Observe that, by Lemma \ref{lemma:rossi}\eqref{it:rossi4},
\begin{equation*}
  \vartheta_1(\iii)=A_{\iii|_n}\vartheta_1(\sigma^n\iii)
\end{equation*}
for all $n \in \N$. Since, by Lemma \ref{lemma:rossi}\eqref{it:rossi5}, $\vartheta_1(\sigma^n\iii)\in \CC$ for all $n\in\N$, we have
\begin{equation*}
    V\in\bigcap_{n=1}^{\infty}\bigcup_{\iii\in\Sigma_n}A_{\iii}\CC
\end{equation*}
as required.

Finally, let us show that $\bigcap_{n=1}^{\infty}\bigcup_{\iii\in\Sigma_n}A_{\iii}\CC \subset Y_F$. To that end, suppose that $V \in \bigcap_{n=1}^{\infty}\bigcup_{\iii\in\Sigma_n}A_{\iii}\CC$. Then for any $n\in\N$, we may choose $\iii_n\in\Sigma_n$ and $V_n\in\CC$, such that $V= A_{\iii_n}V_n$. Let $v_n$ be a unit vector such that $V_n=\projspan{ v_n}$. Note that the set
\begin{equation*}
    S=\{A\in GL_2(\R)\colon \|A\|=1\}
\end{equation*}
is a compact subset of $GL_2(\R)$. By passing to a sub-sequence if necessary, we may assume that $v_n\to v$ for some $v \in S^1$ and
\begin{equation*}
    \frac{A_{\iii_n}}{\|A_{\iii_n}\|}\to A
\end{equation*}
for some $A \in S$. Now, by recalling the definition of domination from \eqref{eq:domination-def}, there exist $C>0$ and $0<\tau<1$ such that
\begin{equation*}
    |\det(\|A_{\iii_n}\|^{-1}A_{\iii_n})|=\frac{\alpha_1(A_{\iii_n})\alpha_2(A_{\iii_n})}{\|A_{\iii_n}\|}=\frac{\alpha_2(A_{\iii_n})}{\alpha_1(A_{\iii_n})}\leq C\tau^{n}.
\end{equation*}
Consequently, $\det(A)=0$, which together with $\|A\|=1$, implies that $\rank(A)=1$. Recall that, by \cite[Lemma 2.3]{BochiMorris2015}, there is a positive constant $\kappa$ such that
\begin{equation*}
    \|A_{\iii_n}v_n\|\geq\kappa\|A_{\iii_n}\|\|v_n\|=\kappa\|A_{\iii_n}\|
\end{equation*}
for all $n \in \N$. Since the maps $A_{\iii_n}$ are linear, it follows from the Banach-Steinhaus theorem that $\|A_{\iii_n}\|^{-1}A_{\iii_n}v_n\to Av$ and therefore,
\begin{equation*}
    \|Av\| = \lim_{n\to\infty}\frac{\|A_{\iii_n}v_n\|}{\|A_{\iii_n}\|} \ge \kappa
\end{equation*}
and $Av$ is a non-zero vector in $\im(A)$. Thus, by the continuity of the map $v\mapsto \projspan{ v}$,
\begin{equation*}
    V=\lim_{n \to \infty}A_{\iii_n}V_n=\lim_{n\to\infty}\projspan{ A_{\iii_n}v_n}=\projspan{ Av}=\im(A)\in Y_F.
\end{equation*}
Therefore, $V \in Y_F$ and the proof is finished.
\end{proof}

\subsection{Bounded neighborhood condition}
To finish this section, we introduce a geometric separation condition for self-affine sets, which we call the bounded neighborhood condition. We remark that a similar condition has already been introduced in \cite{Kaenmaki2004}. We also define a weaker variant which allows exact overlaps in the construction. Let $X$ be a self-affine set and
\begin{equation*}
    \Phi(x,r)= \{\varphi_{\iii}\colon \alpha_2(A_{\iii})\leq r < \alpha_2(A_{\iii^-})\text{ and }\varphi_{\iii}(X)\cap B(x,r)\ne\emptyset\}
\end{equation*}
for all $x \in X$ and $r>0$. We say that $X$ satisfies the \emph{weak bounded neighborhood condition (WBNC)}, if
\begin{equation*}
    \sup_{\atop{x \in X}{r>0}}\#\Phi(x,r) < \infty.
\end{equation*}
Furthermore, $X$ satisfies the \emph{bounded neighborhood condition (BNC)} if it satisfies the WBNC and $\varphi_{\iii} \ne \varphi_{\jjj}$ whenever $\iii,\jjj \in\Sigma_*$ such that $\iii \ne \jjj$. It turns out that if the SSC is not satisfied, then the WBNC is the right separation condition for studying the tangent structure of $X$. Let us comment on how the BNC and the WBNC are related to other separation conditions. It is not difficult to see that the SSC implies the BNC, but we will give an example of a self-affine set satisfying the SOSC but not the BNC later in Example \ref{ex:counterexample}. This also shows that it is not possible to replace the WBNC with the SOSC in the assumptions of the main result of Section \ref{sec:tangent-decompo}.

\section{Tangent decompositions and slices} \label{sec:tangent-decompo}
We begin to study the structure of weak tangent sets of dominated self-affine sets satisfying the bounded neighborhood condition. In the presence of the WBNC, we show the existence of tangent decompositions and demonstrate how they can be used to study slices of the set. Our main observation in this section is the following proposition which generalizes B\'ar\'any, K\"aenm\"aki, and Rossi \cite[Theorem 5.2]{BaranyKaenmakiRossi2021}. By $A+x$ we mean the set $\{a+x\colon a\in A\}$ for all $A\subset \R^2$ and $x\in\R^2$.

\begin{proposition}\label{prop:assouad-upper-bound}
  If $X$ is a self-affine set satisfying the WBNC, then for every $T\in\Tan(X)$ there exist $x\in X$ and $V\in X_F$ such that
  \begin{equation*} 
    \dimh(T)\leq\max\{\dimh(X), 1+\dimh(X\cap (V+x))\}.
  \end{equation*}
  In particular,
  \begin{equation*}
    \dima(X)\leq\max\{\dimh(X), 1+\sup_{\atop{x\in X}{V\in X_F}}\dimh(X\cap (V+x))\}.
  \end{equation*}
  If $X$ is dominated, then $\dimh(X)$ can be removed from both maxima above.
\end{proposition}
In Example \ref{ex:counterexample}, we show that the proposition can fail if the WBNC is not satisfied and in fact, this is even possible under the SOSC. In particular, the previous proposition is not true if one replaces the WBNC by the SOSC.

The proof of Proposition \ref{prop:assouad-upper-bound} relies on finding suitable decompositions of the tangents of self-affine sets into finitely many components, where each component can be affinely mapped to a slice of the original set. This is made formal by the following lemma.
\begin{lemma}\label{lemma:tangent-decomp}
  If $X$ is a self-affine set satisfying the WBNC and $T\in \Tan(X)$, then there exists a finite index set $I$ such that for every $i\in I$ there is a set $T_i\subset T$, a point $y_i\in X$, and a linear map $G_i$ for which
  \begin{enumerate}
      \item\label{it:decomp1} $T=\bigcup_{i\in I}T_i$,
      \item\label{it:decomp2} $\mathrm{rank}(G_i)\geq 1$,
      \item\label{it:decomp3} $G_i(T_i)+y_i\subset X$.
  \end{enumerate}
  Furthermore, if $X$ is dominated, then $\mathrm{rank}(G_i)=1$ and $\im(G_i)\in X_F$ for all $i\in I$.
\end{lemma}

\begin{proof}
Let $T\in\Tan(X)$. By definition, we may choose a sequence $(\iii_n)_{n\in\N}$ of infinite words and a sequence $(r_n)_{n\in\N}$ of positive real numbers converging to $0$ such that
\begin{equation*}
    M_{\pi\iii_n,r_n}(X)\cap B(0,1)\to T
\end{equation*}
in Hausdorff distance. Since $X$ satisfies the WBNC, there exists $M>0$, such that
\begin{equation*}
    \#\Phi(\pi\iii_n,r_n)\leq M,
\end{equation*}
for all $n \in \N$. Hence, there is $K \in \{1,\ldots,M\}$ such that $\#\Phi(\pi\iii_n,r_n)=K$ for infinitely many $n$. In other words, there exists a sequence $(n_k)_{k \in \N}$ of natural numbers such that $\#\Phi(\pi\iii_{n_k},r_{n_k})=K$ for all $k\in\N$. Write
\begin{equation*}
  \Phi(\pi\iii_{n_k},r_{n_k}) = \{\varphi_{\jjj_{n_k}^i}\}_{i=1}^K
\end{equation*}
for all $k \in \N$. By passing to a sub-sequence, if necessary, we see that for every $i\in\{1,\ldots,K\}$, there exists a set $T_i$ such that
\begin{equation*}
  (M_{\pi\iii_{n_k},r_{n_k}} \circ \varphi_{\jjj_{n_k}^1},\ldots,M_{\pi\iii_{n_k},r_{n_k}} \circ \varphi_{\jjj_{n_k}^K})(X^K) \cap B(0,1)^K \to T_1 \times \cdots \times T_K
\end{equation*}
in Hausdorff distance. Noting that
\begin{equation*}
    M_{\pi\iii_{n_k},r_{n_k}}(X)\cap B(0,1)= \bigcup_{i=1}^KM_{\pi\iii_{n_k},r_{n_k}}\circ \varphi_{\jjj_{n_k}^{i}}(X)\cap B(0,1)
\end{equation*}
for all $k \in \N$, we see that \eqref{it:decomp1} holds.

Since $X^K$ is compact, we may assume that
\begin{equation*}
  (\varphi_{\jjj_{n_k}^{1}}^{-1}(\pi\iii_{n_k}),\ldots,\varphi_{\jjj_{n_k}^{K}}^{-1}(\pi\iii_{n_k})) \to (y_1,\ldots,y_K)\in X^K,
\end{equation*}
and therefore for each $i \in \{1,\ldots,K\}$ there exists a linear map $G_i$ such that
\begin{equation*}
    (\varphi_{\jjj_{n_k}^{1}}^{-1}\circ M_{\pi\iii_{n_k},r_{n_k}}^{-1},\ldots,\varphi_{\jjj_{n_k}^{K}}^{-1}\circ M_{\pi\iii_{n_k},r_{n_k}}^{-1}) \to (G_1+y_1,\ldots,G_K+y_K)
\end{equation*}
in the uniform convergence in $X^K$. Clearly, 
\begin{equation*}
   \varphi_{\jjj_{n_k}^{i}}^{-1}\circ M_{\pi\iii_{n_k},r_{n_k}}^{-1}(M_{\pi\iii_{n_k},r_{n_k}}\circ \varphi_{\jjj_{n_k}^{i}}(X)\cap B(0,1))\subset X,
\end{equation*}
so by taking the limit, we see that $G_i(T_i)+y_i\subset X$ which proves \eqref{it:decomp2}.

Finally, to prove \eqref{it:decomp3}, denote by $A_{\jjj_{n_k}^{i}}$ the linear part of $\varphi_{\jjj_{n_k}^{i}}$. Then, by the definition of $\Phi(\pi\iii_{n_k},r_{n_k})$, we have that
\begin{equation*}
    \|r_{n_k}A_{\jjj_{n_k}^{i}}^{-1}\| = r_{n_k}\alpha_2(A_{\jjj_{n_k}^{i}})^{-1} \geq 1.
\end{equation*}
Since $A\mapsto\|A\|$ is continuous, we have $\|G_i\|=\lim_{k\to\infty}\|r_{n_k}A_{\jjj_{n_k}^{i}}^{-1}\|\geq 1>0$ and, in particular, $\mathrm{rank}(G_i)\geq 1$ for all $i \in \{1,\ldots,K\}$.

Let us next assume that $X$ is dominated. Fix $i\in I$ and for simplicity, denote $\jjj_{n_k}^{i}$ by $\jjj_k$. First observe that the sequence $(|\jjj_k|)_{k \in \N}$ is unbounded, since if it was bounded by some number $L\in\N$, we would have
\begin{equation*}
    r_{n_k} \ge \alpha_2(A_{\jjj_k})\geq (\min_{j \in \{1,\ldots,N\}} \alpha_2(A_j))^L > 0
\end{equation*}
for all $k\in\N$ contradicting the fact that $\lim_{k \to \infty} r_{n_k} = 0$. By domination, there exist $C>0$ and $0<\tau<1$ such that
\begin{align*}
    |\det(r_{n_k}A_{\jjj_k}^{-1})|&=\frac{r_{n_k}^2}{\det(A_{\jjj_k})}=\frac{r_{n_k}^2}{\alpha_1(A_{\jjj_k})\alpha_2(A_{\jjj_k})}\\
    &\leq\frac{\alpha_2(A_{\jjj_k})}{\alpha_1(A_{\jjj_k})\min_{j\in\{1,\ldots,N\}}\alpha_2(A_j)} \leq \frac{C}{\min_{j\in\{1,\ldots,N\}}\alpha_2(A_j)}\tau^{|\jjj_k|}.
\end{align*}
Since $|\jjj_k|$ is unbounded, we see that $\mathrm{rank}(G_i)=1$. Finally, since $r_{n_k}A_{\jjj_{n_k}^i}^{-1}$ converges to the linear map $G_i$, which is a rank one map, $\im(G_i)\in X_F$.
\end{proof}

We are now ready to prove Proposition \ref{prop:assouad-upper-bound}.

\begin{proof}[Proof of Proposition \ref{prop:assouad-upper-bound}]
  Let $T\in\Tan(X)$ and $\{T_i\}_{i\in I}$ be a tangent decomposition of $T$ given by Lemma \ref{lemma:tangent-decomp}. Notice that, since $T=\bigcup_{i\in I}T_i$, we have $\dimh(T)=\max_{i\in I}\dimh(T_i)$. Let $i \in I$ be the index which achieves this maximum. By Lemma \ref{lemma:tangent-decomp}, we have $G_i(T_i)+y_i\subset X\cap (\im(G_i)+y_i)$ and hence,
  \begin{equation*}
    \dimh(G_i(T_i)+y_i) \le \dimh(X \cap (\im(G_i)+y_i)).
  \end{equation*}
If $\mathrm{rank}(G_i)=2$ then $\dimh(G_i(T_i)+y_i)\leq\dimh(X)$. On the other hand, if $\mathrm{rank}(G_i)=1$ then $x \mapsto G_ix+y_i$ is bi-Lipschitz equivalent to $\proj_{\ker(G_i)^\bot}$, thus we have
  \begin{align*}
    \dimh(T_i) &\le \dimh(\R \times \proj_{\ker(G_i)^\bot}(T_i)) \\ 
    &= 1+\dimh(\proj_{\ker(G_i)^\bot}(T_i)) = 1+\dimh(G_i(T_i)+x).
  \end{align*}
  Therefore,
  \begin{equation*}
    \dimh(T) = \dimh(T_i) \le 1+\dimh(X \cap (\im(G_i)+y_i))
  \end{equation*}
  and we have shown the first claim. By Lemma \ref{lemma:kaenmakiojalarossi}, the second claim follows immediately from the first claim.

 If $X$ is dominated, then $\mathrm{rank}(G_i)=1$ above by the last assertion of Lemma \ref{lemma:tangent-decomp} and the last claim follows.
\end{proof}

\begin{figure}
    \centering
    \includegraphics[width=0.8\linewidth]{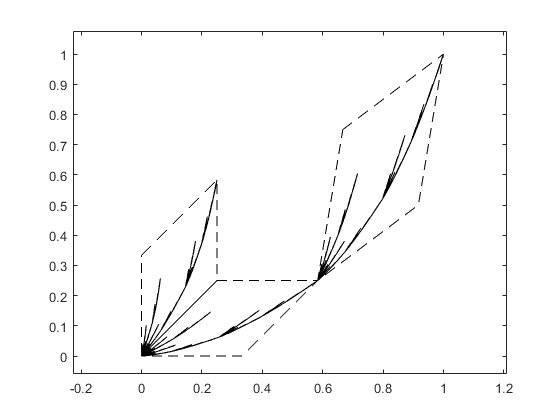}
    \caption{The self-affine set of Example \ref{ex:counterexample}. The first level cylinders are illustrated by dashed lines.}
    \label{fig:counterexample}
\end{figure}

\begin{example}\label{ex:counterexample}
In this example, we exhibit an affine IFS $(\fii_1,\fii_2,\fii_3)$, where $\fii_i(x)=A_ix+b_i$ for all $x \in \R^2$, with self-affine set $X$ satisfying the SOSC but not the WBNC such that $(A_1,A_2,A_3)$ is dominated and 
\begin{equation}\label{eq:counterexample}
1+\sup_{\substack{x\in X\\ V\in X_F}}\dimh(X\cap (V+x))<\dima(X) = 2.
\end{equation}
In particular, this shows that the upper bound of Proposition \ref{prop:assouad-upper-bound} can fail if the WBNC is replaced by the SOSC. 
Let
\begin{equation*}
  A_1=
  \begin{pmatrix}
    \frac{1}{3}& \frac{1}{4}\\
    0&\frac{1}{4}
  \end{pmatrix},\qquad
  A_2=
  \begin{pmatrix}
    \frac{1}{4}& 0\\
    \frac{1}{4}&\frac{1}{3}
  \end{pmatrix},\qquad
  A_3=
  \begin{pmatrix}
    \frac{1}{3}& \frac{1}{12}\\
    \frac{1}{4}&\frac{1}{2}
  \end{pmatrix}
\end{equation*}
and
\begin{equation*}
  b_1=(0,0), \qquad b_2=(0,0), \qquad b_3=(\tfrac{7}{12},\tfrac{1}{4}).
\end{equation*}
For illustration of the associated self-affine set $X$, see Figure \ref{fig:counterexample}. Since $\bar{0}=(0,0)$ is a fixed point for both $\fii_1$ and $\fii_2$, we have $\bar{0} \in \fii_1(X) \cap \fii_2(X)$ and $X$ does not satisfy the SSC. However, $X$ clearly satisfies the SOSC with the open set $U=(0,1)^2$. Furthermore, it is not difficult to see that for any $M\in\N$ there exists $r>0$ such that
\begin{equation*}
    \#\Phi(\bar{0},r)\geq M,
\end{equation*}
so $X$ does not satisfy the WBNC.

For each $\varepsilon\geq0$ let $\CC_\varepsilon\subset\RP$ be the cone having the lines $\langle(1,-\varepsilon)\rangle$ and $\langle(-\varepsilon,1)\rangle$ as boundaries, and containing the line $\langle(1,1)\rangle$. It is easy to see that $\CC_\varepsilon$ is strongly invariant with respect to $(A_1,A_2,A_3)$ for every sufficiently small $\varepsilon>0$ and hence, $(A_1,A_2,A_3)$ is dominated. We also have $A_i\CC_0\subset\CC_0$ for all $i \in \{1,2,3\}$. Let $\mathcal{D}_0 \subset \RP$ be the cone having the lines $\langle(3,-1)\rangle$ and $\langle (1,-3)\rangle$ as boundaries, and containing the line $\langle(1,-1)\rangle$. It is easy to see that $A_i^{-1}\mathcal{D}_0\subset \mathcal{D}_0$ and that $(3,-1)$ and $(1,-3)$ are eigenvectors of $A_1$ and $A_2$, respectively. Therefore, $X_F$ is not a singleton, and furthermore, by Lemma~\ref{lemma:alternate-furstenberg-dir}, $X_F\subset\mathcal{D}_0$. Simple algebraic manipulations show that
\begin{equation}\label{eq:normbound}
    \|A_i\|<0.62,\qquad \min_{V\in\CC_0}\|A_i|_V\|\geq\frac13,\qquad\max_{V\in\mathcal{D}_0}\|A_i^{-1}|_V\|^{-1}\leq\frac{7}{12}\sqrt{\frac{5}{17}}<0.32
\end{equation}
for all $i \in \{1,2,3\}$.

Let us now show that $\dima(X)=2$. As $X \subset \R^2$, it is enough to prove $\dima(X) \ge 2$, and for this, by recalling Lemma \ref{lemma:kaenmakiojalarossi}, we construct a suitable weak tangent. Let us define $T_n \colon \R^2 \to \R^2$ by setting
\begin{equation*}
    T_n(x)=M_{\bar{0},3^{-n}}(x)=3^nx
\end{equation*}
for all $x \in \R^2$ and $n\in\N$, and let $T$ be the Hausdorff limit of the sequence $T_n(X)\cap B(0,1)$. If we can show that $T = Q_1$, where $Q_1 = B(0,1)\cap [0,1]^2$ is the closed first quadrant of the unit ball, then $\dima(X) \ge \dimh(T) \ge 2$ as required. To that end, consider the set $\Sigma^{1,2}_n= \{i_1\cdots i_n\in\Sigma_n\colon i_k\in\{1,2\} \text{ for all }k \in \{1,\ldots,n\}\}$. Let $e_1=(1,0)$ and $e_2=(0,1)$. Note that, since $\|A_{\iii}e_k\| \ge 3^{-n}$ for both $k \in \{1,2\}$ and $\iii\in\Sigma_n^{1,2}$ by \eqref{eq:normbound}, we have
$$
  T_n(\varphi_{\iii}(U))\cap T_n(\varphi_{\jjj}(U))=\emptyset\qquad\text{and}\qquad\bigcup_{\iii\in\Sigma_n^{1,2}} T_n(\varphi_{\iii}(\overline{U}))\cap B(0,1)=Q_1,
$$  
for all $\iii,\jjj\in\Sigma_n^{1,2}$ with $\iii\neq\jjj$. Then for any $\iii\in\Sigma_n^{1,2}$, the central angle $\alpha_\iii$ of the sector $T_n(\varphi_{\iii}(U))\cap B(0,1)$ is 
\begin{equation}\label{eq:central-angle}
  \alpha_{\iii}=\sphericalangle(\langle A_{\iii}e_1\rangle,\langle A_{\iii}e_2\rangle)=\arcsin\left(\frac{|\det(A_{\iii})|}{\| A_{\iii}e_1\|\| A_{\iii}e_2\|}\right)\leq\arcsin((\tfrac{3}{4})^n),
\end{equation}
where in the last inequality, we used \eqref{eq:normbound}. Let $y\in Q_1$ and notice that for every $n\in\N$ there exists $\iii\in\Sigma^{1,2}_n$ such that $y\in T_n(\varphi_{\iii}(\overline{U}))$. Since $X$ contains a continuous path between the points $(0,0)$ and $(1,1)$, there is a point $x\in\partial B(0,1)\cap T_n(\varphi_{\iii}(U))$ such that $T_n(\varphi_{\iii}(X))\cap B(0,1)$ contains a continuous path between the points $(0,0)$ and $x$. In particular, together with \eqref{eq:central-angle}, this implies that there exists a point $z_n\in T_n(\varphi_{\iii}(X))\cap B(0,1)$ such that
\begin{equation*}
    |y-z_n|\leq \arcsin((\tfrac{3}{4})^n).
\end{equation*}
Therefore, $T_n(X)\cap B(0,1)\to Q_1$ finishing the proof of $\dima(X)\geq 2$.

It suffices to show that there exists $c<1$ such that $\dimh(X\cap (V+x))\leq c$ for all $x\in X$ and $V\in X_F$. Fix $x\in X$ and $V\in X_F$, and notice that, by \eqref{eq:normbound}, we have
\[
\begin{split}
\diam(\varphi_{\iii}(\overline{U})\cap(V+x))&=\|A_{\iii}|_{A_{\iii}^{-1}V}\|\diam(\overline{U}\cap(A_{\iii}^{-1}V+\varphi_{\iii}^{-1}(x)))\\
&\leq\|A_{\iii}|_{A_{\iii}^{-1}V}\|\sqrt{2}=\|A_{\iii}^{-1}|_V\|^{-1}\sqrt{2}\leq(0.32)^n\sqrt{2}
\end{split}
\]
for all $n\in\N$ and $\iii\in\Sigma_n^{1,2}$. Hence, for each $s>-\frac{\log 3}{\log0.32}$, we see that
\[
\begin{split}
\mathcal{H}^s(X\cap(V+x))\leq\lim_{n\to\infty}2^{s/2}\sum_{\iii\in\Sigma_n}\|A_{\iii}^{-1}|_V\|^{-s}\leq \lim_{n\to\infty}2^{s/2}3^n(0.32)^{sn}=0,
\end{split}
\]
and so $\dimh(X\cap (V+x))\leq-\frac{\log 3}{\log0.32}<1$.
\end{example}

\begin{remark}
    For the purpose of this remark, let us briefly recall some definitions. For each $A \in GL_2(\R)$ and $s \ge 0$, the \emph{singular value function} is
    \begin{equation*}
        \fii^s(A) =
        \begin{cases}
            \alpha_1(A)^s, &\text{if } 0 \le s \le 1, \\
            \alpha_1(A)\alpha_2(A)^{s-1}, &\text{if } 1 < s \le 2, \\
            (\alpha_1(A)\alpha_2(A))^{s/2}, &\text{if } s > 2.
        \end{cases}
    \end{equation*}
    The value $\fii^s(A)$ represents a measurement of the $s$-dimensional volume of the ellipse $A(B(0,1))$. For each $\A \in GL_2(\R)^N$ and $s \ge 0$, the \emph{pressure} is
    \begin{equation*}
        P(\A,s) = \lim_{n \to \infty} \frac{1}{n}\log\sum_{\iii \in \Sigma_n} \fii^s(A_\iii).
    \end{equation*}
    As the singular value function is sub-multiplicative, the limit above exists by Fekete's lemma. It is also easy to see that the pressure $P(\A,s)$ is continuous and strictly decreasing as a function of $s$ with $P(\A,0) \ge 0$ and $\lim_{s \to \infty}P(\A,s) = -\infty$. We may thus define the \emph{affinity dimension} by setting $\dimaff(\A)$ to be the minimum of $2$ and the unique $s \ge 0$ for which $P(\A,s)=0$. If $X$ is a self-affine set, then $\dimaff(X)$ denotes the affinity dimension of the associated tuple of matrices.
    Also recall that a self-affine set is \emph{strongly irreducible} if no finite collection of lines in $\RP$ is preserved by all of the matrices in the tuple.
    
    In \cite[Example 3.3]{BaranyKaenmakiYu2021-preprint}, the authors answer a question posed by Fraser in \cite{Fraser2020}, by constructing an example of a self-affine set $X$ satisfying $\diml(X)< \dimh(X)=\dimaff(X)< \dima(X)$, where $\diml$ denotes the \emph{lower dimension}; see  \cite[\S 3.1]{Fraser2020} for the definition. The construction is strongly based on the properties of an underlying self-affine carpet, so it is an interesting question whether this behaviour is possible when $X$ has no reducible subsystems. Let $X$ be the self-affine set defined in Example \ref{ex:counterexample}. By an argument similar to the calculation of the Assouad dimension in Example \ref{ex:counterexample}, it is easy to see that $X$ has a line segment as a weak tangent at the point $(1,1)$, and therefore $\diml(X)\le1$ by \cite[Theorem 1.1]{FHKY2019}. Moreover, since the matrices $(A_1,A_2,A_3)$ have pairwise distinct eigenvectors, the strong irreducibility follows and by \cite[Theorem 1.1]{BaranyHochmanRapaport2019} and a simple calculation using \eqref{eq:normbound}, we have $1<\dimh(X)=\dimaff(X)<2$. Thus, Example \ref{ex:counterexample} shows that, in the absence of strong separation, the strict inequalities $\diml(X)< \dimh(X)=\dimaff(X)< \dima(X)$ are possible for self-affine systems with no reducible subsystems.
\end{remark}

\section{Self-affine sets with large projections} \label{sec:self-affine-large}
In this section, we show that if all the projections of the self-affine set have maximal dimension, then we have equality in Proposition \ref{prop:assouad-upper-bound}. We also show that the supremum in the statement can be replaced by a maximum. The following theorem is the main result of this section and Proposition \ref{cor:sosc-assouad} below assures that it generalizes B\'ar\'any, K\"aenm\"aki, and Yu \cite[Theorem 3.2]{BaranyKaenmakiYu2021-preprint} by relaxing the SSC to a separation condition which allows slight overlapping.

\begin{theorem}\label{thm:dimasosc}
  Let $X$ be a dominated self-affine set satisfying the WBNC such that $\dimh(X)\geq 1$ and $\dimh(\proj_{V^{\perp}}(X))=1$ for all $V \in X_F$, then
  \begin{align*}
    \dima(X)&=1+\max_{\atop{x\in X}{V\in X_F}}\dimh(X\cap(V+x))\\ 
    &=1+\max_{\atop{x\in X}{V\in \RP\setminus Y_F}}\dima(X\cap(V+x)).
  \end{align*}
\end{theorem}

The proof of the theorem uses ideas introduced in \cite[\S 5]{BaranyKaenmakiYu2021-preprint}, but the absence of strong separation induces some complications. We essentially split \cite[Lemma 5.2]{BaranyKaenmakiYu2021-preprint}, which assumes the SSC, into Lemmas \ref{lemma:dim-1-slice} and \ref{lemma:small-slice} and make two key observations to work around the lack of SSC. First of all, the intuition behind \cite[Lemma 5.2]{BaranyKaenmakiYu2021-preprint} is that the weak tangent sets of the self-affine set $X$ have a comb-like structure, where the slices of the tangent set along the direction of the teeth of the comb have full dimension, and the dimensions of the slices in directions perpendicular to the teeth have dimension comparable to some slice of the self-affine set in a Furstenberg direction. By Lemma \ref{lemma:tangent-decomp}, we know that under the bounded neighborhood condition, the situation is similar in the sense that the weak tangents are finite unions of these comb-like sets. Secondly, in \cite[\S 5]{BaranyKaenmakiYu2021-preprint}, to show that the teeth of the combs point to the same direction, the authors use the fact that any slice of the self-affine set has dimension strictly smaller than one, which does not have to be true in our setting. We work around this using domination in the following lemma.

\begin{lemma}\label{lemma:convergence}
  Let $X$ be a dominated self-affine set and $(\iii_k)_{k \in \N}$ be a sequence of infinite words in $\Sigma$. If $(n_k)_{n \in \N}$ is an increasing sequence of integers such that the limit $\lim_{k\to\infty}\vartheta_1(\iii_k|_{n_k})$ exists, then $\lim_{k\to\infty}\overline{\vartheta}_1(\iii_k)$ exists and
  \begin{equation*}
      \lim_{k\to\infty}\overline{\vartheta}_1(\iii_k) = \lim_{k\to\infty}\vartheta_1(\iii_k|_{n_k}).
  \end{equation*}
\end{lemma}

\begin{proof}
Let $(n_k)_{k \in \N}$ be a strictly increasing sequence of integers such that the limit $W \coloneqq \lim_{k\to\infty}\vartheta_1(\iii_k|_{n_k}) \in \RP$ exists and let $\varepsilon>0$. By Lemma \ref{lemma:rossi}\eqref{it:rossi1}, $\overline{\vartheta}_1(\iii_k)$ is well defined for every $k\in\N$ and, by the uniform convergence, we may choose $k_0\in\N$ large enough such that
\begin{equation*}
    \sphericalangle(\vartheta_1(\iii_k|_{n_k}),\overline{\vartheta}_1(\iii_k)) < \frac{\varepsilon}{2}
\end{equation*}
for all $k\geq k_0$. By the assumption, by making $k_0$ larger if needed, we also have
\begin{equation*}
    \sphericalangle(\vartheta_1(\iii_k|_{n_k}),W) < \frac{\varepsilon}{2},
\end{equation*}
for all $k\geq k_0$. Thus, by the triangle inequality, we have
\begin{equation*}
    \sphericalangle(\overline{\vartheta}_1(\iii_k),W)\leq \sphericalangle(\overline{\vartheta}_1(\iii_k),\vartheta_1(\iii_k|_{n_k}))+\sphericalangle(\vartheta_1(\iii_k|_{n_k}),W) < \varepsilon
\end{equation*}
and therefore, $\lim_{k\to\infty}\overline{\vartheta}_1(\iii_k)=W$.
\end{proof}

We abuse notation by denoting the intersection of $T\in\Tan(X)$ with the open unit ball by $T^{\circ}$. Similarly, if $\{T_i\}_{i\in I}$ is a tangent decomposition of $T$, then we let $T_i^{\circ}=T_i\setminus \partial B(0,1)$. This should not cause any confusion, since we will not be referring to the actual interior of $T$ at any point. Furthermore, by the rank, image, and kernel of an affine map, we mean the rank, image, and kernel of its linear part.

\begin{lemma}\label{lemma:dim-1-slice}
  Let $X$ be a dominated self-affine set satisfying the WBNC such that, $\dimh(X)\geq 1$, and $\dimh(\proj_{V^{\perp}}(X))=1$ for all $V \in X_F$. Let $T\in\Tan(X)$ and $\{T_i\}_{i\in I}$ be a tangent decomposition of $T$ given by Lemma \ref{lemma:tangent-decomp}. Then for every $i \in I$ there exists $W_i\in Y_F$ such that
  \begin{equation*}
      \dimh(T\cap(W_i+y))= 1
  \end{equation*}
  for all $y\in T_i^{\circ}$.
\end{lemma}

\begin{proof}
Let $(\iii_k)_{k \in \N}$ be a sequence of infinite words in $\Sigma$ and $(r_k)_{k \in \N}$ be a sequence of positive real numbers converging to zero such that
\begin{equation*}
    M_{\pi\iii_k,r_k}(X)\cap B(0,1)\to T.
\end{equation*}
Recall from the proof of Lemma \ref{lemma:tangent-decomp} that there exists a sequence $(n_k)_{k \in \N}$ of integers and, for each $i \in I$, finite words $\jjj_{n_k}^i \in \Sigma_*$ and sets $T_i$ such that
\begin{equation*}
  M_{\pi\iii_{n_k},r_{n_k}} \circ \varphi_{\jjj_{n_k}^i}(X) \cap B(0,1) \to T_i
\end{equation*}
for all $i \in I$ in Hausdorff distance. Fix $y\in T^{\circ}$ and choose $i\in I$ such that $y\in T_i^{\circ}$. Since $y\not\in\partial B(0,1)$, there is $\delta>0$ depending only on $y$ such that $B(y,2\delta)\subset B(0,1)$. Therefore, there are infinite words $\jjj_k\in[\jjj_{n_k}^i]$ such that $M_{\pi\iii_{n_k},r_{n_k}}(\pi(\jjj_k))\to y$ and
\begin{equation}\label{eq:Py-inclusion}
    M_{\pi\iii_{n_k},r_{n_k}}(X\cap B(\pi\jjj_k,\delta r_{n_k}))\subset M_{\pi\iii_{n_k},r_{n_k}}\circ \varphi_{\jjj_{n_k}^{i}}(X)\cap B(0,1)
\end{equation}
for all large enough $k\in\N$. Let $m_k\geq n_k$ be the unique integer which satisfies
\begin{equation}\label{eq:rk_estimate}
    \alpha_1(A_{\jjj_k|_{m_k}})\leq \delta r_{n_k}< \alpha_1(A_{\jjj_k|_{m_k-1}}).
\end{equation}
By again passing to a sub-sequence, if necessary, there exists an affine map $P_y$ such that
\begin{equation*}
    M_{\pi\iii_{n_k},r_{n_k}}\circ\varphi_{\jjj_k|_{m_k}}\to P_y
\end{equation*}
in the uniform convergence in $X$. By compactness and \eqref{eq:Py-inclusion}, we have $y \in P_y(X) \subset T_i$ and, by domination, we have
\begin{equation*}
    \frac{\alpha_2(A_{\jjj_k|_{m_k}})}{\alpha_1(A_{\jjj_k|_{m_k}})}\leq C\tau^{m_k},
\end{equation*}
so in particular $\det(r_{n_k}^{-1}A_{\jjj_k|_{m_k}})\to 0$ as $k\to\infty$. Also, by (\ref{eq:rk_estimate}) and \cite[Corollary 2.4]{BaranyKaenmakiMorris2018}, there exists a constant $C>0$ such that we have $\|r_{n_k}^{-1}A_{\jjj_k|_{m_k}}\|\geq C\delta$ for all $k\in\N$. Therefore, we see that $\rank(P_y)=1$. Let $W_y=\im(P_y)$ and note that by Lemma \ref{lemma:alternate-furstenberg-dir}, $W_y\in Y_F$. Recall that $P_y(X)$ and $\proj_{\ker(P_y)^{\perp}}(X)$ are bi-Lipschitz equivalent, so by the assumption,
\begin{equation}\label{eq:dim-projection}
\begin{split}
    \dimh(T\cap (W_y+y))&\geq \dimh(P_y(X)\cap (W_y+y)) \\
    &\geq \dimh(\proj_{\ker(P_y)^{\perp}}(X)).
\end{split}
\end{equation}
Let us show that $\ker(P_y)\in X_F$. Observe that the linear part of the map $M_{\pi\iii_{n_k},r_{n_k}}\circ\varphi_{\jjj_k|_{m_k}}$ is $r_{n_k}^{-1}A_{\jjj_k|_{m_k}}$ and notice that this sequence converges to the linear part of $P_y$ in the uniform convergence in $X$. Denote by $A_y$ the linear part of $P_y$ and let $v$ be a unit vector in the kernel of $A_y$. Since the eigenspaces of $(r_{n_k}^{-1}A_{\jjj_k|_{m_k}})^T(r_{n_k}^{-1}A_{\jjj_k|_{m_k}})$ converge to the eigenspaces of $A_y^TA_y$ and since $\ker(A_y)=\ker(A_y^TA_y)$ is the eigenspace corresponding to the singular value $0$, we see that there is a sequence of unit vectors $v_k\to v$, such that
\begin{equation*}
    A_{\jjj_k|_{m_k}}^TA_{\jjj_k|_{m_k}}v_k=\alpha_2(A_{\jjj_k|_{m_k}})^2v_k
\end{equation*}
for all $k\in\N$. Let us define $w_k= \alpha_2(A_{\jjj_k|_{m_k}})^{-1}A_{\jjj_k|_{m_k}}v_k$. By the previous, we have
\begin{align*}
    \|w_k\|&=\alpha_2(A_{\jjj_k|_{m_k}})^{-1}\|A_{\jjj_k|_{m_k}}v_k\|=\alpha_2(A_{\jjj_k|_{m_k}})^{-1}(A_{\jjj_k|_{m_k}}v_k\,\cdot\,A_{\jjj_k|_{m_k}}v_k)^{\frac{1}{2}}\\
    &=\alpha_2(A_{\jjj_k|_{m_k}})^{-1}\langle A_{\jjj_k|_{m_k}}^TA_{\jjj_k|_{m_k}}v_k\,|\,v_k\rangle^{\frac{1}{2}}=\alpha_2(A_{\jjj_k|_{m_k}})^{-1}(\alpha_2(A_{\jjj_k|_{m_k}})^{-1}v_k\,\cdot\,v_k)^{\frac{1}{2}}\\
    &=(v_k\,\cdot\,v_k)^{\frac{1}{2}}=\|v_k\|=1,
\end{align*}
where $\cdot$ is the standard inner product on $\R^2$. Therefore, by possibly passing to a sub-sequence, we may assume that $w_k$ converges to some unit vector $w$ and that
\begin{equation*}
    \frac{(A_{\jjj_k|_{m_k}})^{-1}}{\|(A_{\jjj_k|_{m_k}})^{-1}\|}\to B_y,
\end{equation*}
for some $2\times2$ matrix $B_y$. Since $\A^{-1}=(A_1^{-1},\ldots,A_N^{-1})$ is dominated, we see as before, that $\det(\|(A_{\jjj_k|_{m_k}})^{-1}\|^{-1}(A_{\jjj_k|_{m_k}})^{-1})\to 0$, so $B_y$ has rank at most one. Since
\begin{equation*}
    \frac{(A_{\jjj_k|_{m_k}})^{-1}}{\|(A_{\jjj_k|_{m_k}})^{-1}\|}\cdot\frac{A_{\jjj_k|_{m_k}}v_k}{\alpha_2(A_{\jjj_k|_{m_k}})}=v_k,
\end{equation*}
by taking limits, we have that $B_y(w)=v$, so $v$ is a unit vector in the image of $B_y$. Therefore $B_y$ is a rank 1 matrix and $\ker(A_y)=\projspan{v}=\im(B_y)\in X_F$. By the assumption and \eqref{eq:dim-projection}, we have that $\dimh(T\cap (W_y+y))\geq 1$. The upper bound is trivial, since $T\cap (W_y+y)$ is contained in a line.

Finally, let us show that $W_y$ is constant in $T_i^{\circ}$. By passing to a sub-sequence, if necessary, we may assume that $\lim_{k\to\infty}\vartheta_1(\jjj_{n_k}^i)=W_i$ for some $W_i\in Y_F$. Notice that $W_i$ does not depend on the choice of $y\in T_i^{\circ}$. By Lemma \ref{lemma:alternate-furstenberg-dir} (or rather its proof), it is easy to see that
\begin{equation*}
    W_y=\im(P_y)=\lim_{k\to\infty}\vartheta_1(\jjj_k|_{m_k}).
\end{equation*}
Therefore, by Lemma \ref{lemma:convergence}, we have $\im(P_y)=\lim_{k\to\infty}\overline{\vartheta}_1(\jjj_k)$. Noting that $\jjj_{n_k}^i=\jjj_k|_{|\jjj_{n_k}^i|}$ and, applying Lemma \ref{lemma:convergence} again, we see that
\begin{equation*}
    W_y=\lim_{k\to\infty} \overline{\vartheta}_1(\jjj_k)=\lim_{k\to\infty} \vartheta_1(\jjj_{n_k}^i)=W_i
\end{equation*}
finishing the proof.
\end{proof}

\begin{lemma}\label{lemma:small-slice}
Let $X\subset \R^2$ be compact. Then for every $x\in X$ and $V\in\RP$ there exist $T\in\Tan(X)$ such that
\begin{equation*}
    \dimh(T\cap V) \geq \dima(X\cap(V+x)) \geq \dimh(X\cap(V+x)).
\end{equation*}
\end{lemma}

\begin{proof}
Let $V\in\RP$ and $x\in X$ and, by recalling Lemma \ref{lemma:kaenmakiojalarossi}, let $T_{\max}\in\Tan(X\cap(V+x))$ be a weak tangent, which satisfies $\dimh(T_{\max})=\dima(X\cap(V+x))$. Let $(x_k)_{k \in \N}$ be a sequence of points in $X\cap(V+x)$ and $(r_k)_{k \in \N}$ be a sequence of positive real numbers converging to zero such that
\begin{equation*}
    M_{x_k,r_k}(X\cap(V+x))\cap B(0,1)\to T_{\max}
\end{equation*}
in Hausdorff distance. Since $x_k\in V+x$ for all $k\in\N$, and each $M_{x_k,r_k}$ is a similarity, we have $M_{x_k,r_k}(V+x)\to V$. Let $T$ be an accumulation point of the sequence $M_{x_k,r_k}(X)\cap B(0,1)$. Then $T\in\Tan(X)$ and, by compactness, $T_{\max}\subset T\cap V$. Therefore, we have
\begin{equation*}
    \dimh(X\cap (V+x))\leq\dima(X\cap (V+x))=\dimh(T_{\max})\leq\dimh(T\cap V)
\end{equation*}
as required.
\end{proof}

We are now ready to prove the main theorem of this section.

\begin{proof}[Proof of Theorem \ref{thm:dimasosc}]
By Lemma \ref{lemma:kaenmakiojalarossi}, we have
\begin{equation}\label{eq:dima-max-tangent}
    \dima(X)=\max_{T\in\Tan(X)}\dimh(T),
\end{equation}
so we may choose $T_{\max}\in\Tan(X)$ such that $\dima(X)=\dimh(T_{\max})$. Recalling Proposition \ref{prop:assouad-upper-bound}, there are $x\in X$ and $V\in X_F\subset \RP\setminus Y_F$ such that
\begin{equation}\label{eq:max-slice}
    \dima(X) = \dimh(T_{\max}) \leq 1+ \dimh(X\cap (V+x)).
\end{equation}
By Lemma \ref{lemma:small-slice}, there exists a tangent set $T$ such that
\begin{equation*}
    \dimh(T\cap V) \geq \dima(X\cap(V+x)) \geq \dimh(X\cap(V+x)).
\end{equation*}
If $\dimh(T\cap V)=0$, then trivially $\dimh(T)\geq 1+\dimh(X\cap (V+x))=1$, by Lemma \ref{lemma:dim-1-slice}. Therefore we may assume that $\dimh(T\cap V)>0$. Notice that $T\cap V\cap \partial B(0,1)$ consists of at most two points, so $\dimh(T^{\circ}\cap V)=\dimh(T\cap V)$. Let $0<s<\dimh(T\cap V)$ and let $\mu$ be a Frostman measure on $T^{\circ}\cap V$; see \cite[Theorem 8.8]{Mattila1995}. Let $\{T_i\}_{i\in I}$ be a tangent decomposition of $T$ given by Lemma \ref{lemma:tangent-decomp}. Since $T=\bigcup_{i\in I}T_i$, at least one of the sets $T_i^{\circ}\cap V$ has positive $\mu$-measure. Let $T_i^{\circ}$ be such a set and let $W_i\in Y_F$ be the line given by Lemma \ref{lemma:dim-1-slice}. Since $V\not\in Y_F$, we have $V\ne W_i$ and so, by the Marstrand's slicing theorem \cite[Theorem 3.3.1]{BishopPeres2017},
\begin{equation*}
    1=\dimh(T\cap (W_i+y))\leq \dimh(T)-s
\end{equation*}
for $\mu$-almost every $y\in T_i^{\circ}\cap V$. In particular, since $\mu(T_i^{\circ}\cap V)>0$, such a point $y$ exists. By letting $s\uparrow \dimh(T\cap V)$, we get
\begin{align*}
    \dima(X) &\geq\dimh(T)\geq 1+\dimh(T\cap V) \\ 
    &\geq 1+\dima(X\cap (V+x)) \geq 1+\dimh(X\cap (V+x)).
\end{align*}
Combining this with (\ref{eq:max-slice}), we get
\begin{equation*}
    \dima(X) = 1+\dima(X\cap (V+x)) = 1+\dimh(X\cap(V+x))
\end{equation*}
as claimed.

It remains to show that $\dima(X\cap(W+y))\leq \dimh(X\cap(V+x))$ for all $y\in X$ and $W\in \RP\setminus Y_F$. By repeating the above proof for such $y$ and $W$, we find a tangent set $T$ such that
\begin{equation*}
    1+\dima(X\cap(W+y))\leq \dimh(T)\leq \dimh(T_{\max})\leq1+ \dimh(X\cap(V+x)),
\end{equation*}
where we used (\ref{eq:dima-max-tangent}) in the middle inequality. This finishes the proof.
\end{proof}

To finish this section, let us verify that Theorem \ref{thm:dimasosc} generalizes B\'ar\'any, K\"aenm\"aki, and Yu \cite[Theorem 3.2]{BaranyKaenmakiYu2021-preprint}. The proof is based on B\'ar\'any, Hochman, and Rapaport \cite[Proposition 6.6]{BaranyHochmanRapaport2019}.

\begin{proposition}\label{cor:sosc-assouad}
  If $X$ is a dominated self-affine set satisfying the WBNC and the SOSC such that $\dimh(X)\geq 1$ and $X_F$ is not a singleton, then
  \begin{equation*}
      \dima(X)=1+\max_{\atop{x\in X}{V\in \RP\setminus Y_F}}\dimh(X\cap (V+x)).
  \end{equation*}
\end{proposition}

\begin{proof}
Since $X$ satisfies the SOSC, \cite[Theorem 2.18]{BaranyKaenmakiYu2021-preprint} shows that
\begin{equation*}
  \dimh(\proj_{V^\bot}(X)) = 1
\end{equation*}
for all $V \in \RP \setminus \mathcal{I}$, where $\mathcal{I} = \{W \in \RP \colon W = A_iW \text{ for all } i \in \{1,\ldots,N\} \}$ and contains at most one element. If $\mathcal{I}=\emptyset$, then the claim follows from Theorem \ref{thm:dimasosc}. Bárány, Käenmäki and Yu \cite[Lemma 2.11]{BaranyKaenmakiYu2021-preprint} show that if $X_F$ is not a singleton, then $\mathcal{I}$ is non empty if and only if the matrices $A_i$ are of the form
\begin{equation*}
    A_i=\begin{pmatrix}
      a_i&b_i\\
      0&d_i
    \end{pmatrix},
\end{equation*}
possibly after a change of basis, where $0<|d_i|<|a_i|<1$, and the matrices are not simultaneously diagonalizable, and clearly in this case, $\mathcal{I}=Y_F$. Since $X_F\subset \RP\setminus Y_F$, Theorem \ref{thm:dimasosc} gives the claim.
\end{proof}

\section{Assouad dimension of the Takagi function} \label{sec:proof-main1}

\begin{figure}
    \centering
    \includegraphics[width=0.8\linewidth]{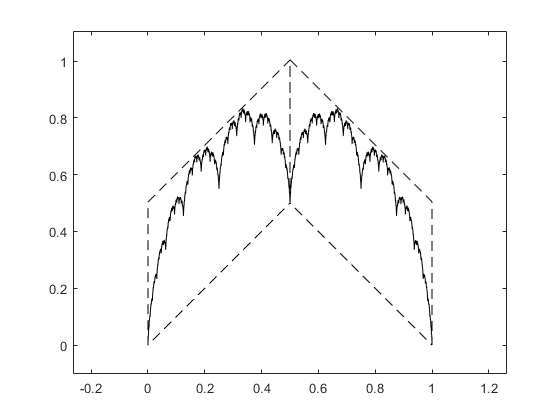}
    \caption{The graph of the Takagi function for $\lambda=\frac{2}{3}$.}
    \label{fig:takagi}
\end{figure}

As an application of Theorem \ref{thm:dimasosc}, which connects the Hausdorff dimension of the slices to the Assouad dimension of the set, we are now able to study slices of the Takagi function. The following result, which is the first part of Theorem \ref{thm:main1}, follows immediately from Theorem \ref{thm:dimasosc} after verifying the assumptions of the theorem.

\begin{theorem}\label{prop:takagi-assouad-formula}
  If $T_\lambda$ is the Takagi function, then
  \begin{align*}
    \dima(\Xl)&=1+\max_{\atop{x\in T_\lambda}{V \in X_F}}\dimh(\Xl\cap (V+x))\\
    &=1+\max_{\atop{x\in T_\lambda}{V\in\RP}}\dima(\Xl\cap (V+x)).
  \end{align*}
\end{theorem}

The second result of this section, which implies the second part of Theorem \ref{thm:main1}, gives an explicit upper bound for the Assouad dimension of the Takagi function. In particular, it follows that the Assouad dimension is always strictly smaller than $2$.

\begin{theorem}\label{thm:takagi-assouad}
  If $T_\lambda$ is the Takagi function, then
  \begin{equation*}
      \max_{\atop{x\in T_\lambda}{V\in\RP}}\dimh(\Xl\cap (V+x)) \leq \frac{\log (2^{n_{\lambda}}-1)}{\log 2^{n_{\lambda}}}<1,
  \end{equation*}
  where
  \begin{equation*}
      n_{\lambda}=\left\lceil\frac{\log2(K_{\lambda}+M_{\lambda})}{-\log \lambda}\right\rceil \ge 2,
  \end{equation*}
  $K_{\lambda}=\sum_{k\in\N}2^{-k}\lambda^{-k}=(2\lambda-1)^{-1}$, and $M_\lambda = \max_{x \in [0,1]}T_\lambda(x)=(3(1-\lambda))^{-1}$.
\end{theorem}

The prerequisite in the proof of the above theorems is to express the Takagi function as a self-affine set. Let $T_{\lambda} \colon [0,1] \to \R$ be the Takagi function for the parameter $\tfrac12 < \lambda < 1$ as defined in \eqref{eq:takagi-def}. Let $\A = (A_1,A_2) \in GL_2(\R)^2$, where
\begin{equation*}
  A_1 =
  \begin{pmatrix}
    \tfrac12 & 0 \\ 
    \tfrac12 & \lambda
  \end{pmatrix}
  \qquad\text{and}\qquad
  A_2 =
  \begin{pmatrix}
    \tfrac12 & 0 \\ 
    -\tfrac12 & \lambda
  \end{pmatrix},
\end{equation*}
and observe that, as $\tfrac12 < \lambda$, both matrices have two real eigenvalues with different absolute values. Furthermore,  the contraction by $\lambda$ is realized precisely on the $y$-axis which is invariant under both matrices. We define affine maps $\fii_1,\fii_2 \colon \R^2 \to \R^2$ by setting
\begin{equation*}
  \fii_1(x) = A_1(x)
  \qquad\text{and}\qquad
  \fii_2(x) = A_2(x)+(\tfrac12,\tfrac12),
\end{equation*}
for all $x \in \R^2$. A straightforward calculation shows that
\begin{equation*}
    T_\lambda(\tfrac{x}{2})=\tfrac{x}{2}+\lambda T_\lambda(x),\text{ and }T_\lambda(\tfrac{x}{2}+\tfrac{1}{2})=\tfrac{1}{2}-\tfrac{x}{2}+\lambda T_\lambda(x).
\end{equation*}
 It follows that $\varphi_1(x,T_\lambda(x))=(\tfrac{x}{2},T_\lambda(\tfrac{x}{2}))$ and $\varphi_2(x,T_\lambda(x))=(\tfrac{x}{2}+\tfrac{1}{2},T_\lambda(\tfrac{x}{2}+\tfrac{1}{2}))$, so $\Xl \subset \R^2$ is the self-affine set associated to the affine IFS $(\fii_1,\fii_2)$. Observe that, by induction, we have
\begin{equation}\label{eq:takagi-matrix}
\begin{split}
    A_{\iii} &=
    \begin{pmatrix}
      2^{-|\iii|} & 0 \\
      \sum_{k=1}^{|\iii|}(-1)^{i_{|\iii|-k+1}+1}2^{-k}\lambda^{|\iii|-k} & \lambda^{|\iii|}
    \end{pmatrix}, \\
    A^{-1}_{\overleftarrow{\iii}} &=
    \begin{pmatrix}
      2^{|\iii|} & 0 \\
      \sum_{k=1}^{|\iii|}(-1)^{i_k}2^{|\iii|-k}\lambda^{-k} & \lambda^{-|\iii|}
    \end{pmatrix}.
\end{split}
\end{equation}
for all $\iii\in\Sigma_*$. We begin verifying the assumptions of Theorem \ref{thm:dimasosc} by showing that $T_\lambda$ is dominated.

\begin{lemma}\label{lemma:takagi-dominated}
  There exists $C>1$ such that
  \begin{equation*}
      \lambda^{|\iii|}\leq\alpha_1(A_{\iii})\leq C\lambda^{|\iii|}
      \qquad\text{and}\qquad
      C^{-1}2^{-|\iii|}\leq\alpha_2(A_{\iii})\leq 2^{-|\iii|}
  \end{equation*}
  for all $\iii\in\Sigma_*$. In particular, $\Xl$ is dominated.
\end{lemma}

\begin{proof}
Let $\iii\in\Sigma_*$ and recall that $\alpha_1(A_{\iii})=\|A_{\iii}\|$. The lower bound for $\alpha_1(A_{\iii})$ follows from the fact that $\lambda^{|\iii|}$ is an eigenvalue of $A_{\iii}$. Similarly, since $\alpha_2(A_{\iii})=\|A_{\iii}^{-1}\|^{-1}$, the upper bound for $\alpha_2(A_{\iii})$ is trivial as $2^{|\iii|}$ is an eigenvalue of $A_{\iii}^{-1}$. We prove the upper bound for $\alpha_1(A_{\iii})$ and only remark that the proof of the lower bound for $\alpha_2(A_{\iii})$ follows similarly. Let $s(\iii)=\sum_{k=1}^{|\iii|}(-1)^{i_{|\iii|-k}+1}2^{-k}\lambda^{|\iii|-k}$ and notice that
\begin{equation*}
    |s(\iii)|\leq \sum_{k=1}^{|\iii|}2^{-k}\lambda^{|\iii|-k}\leq\lambda^{|\iii|}\sum_{k=1}^{\infty}2^{-k}\lambda^{-k}=K_{\lambda}\lambda^{|\iii|}
\end{equation*}
for all $\iii = i_1 \cdots i_{|\iii|} \in \Sigma_*$. Writing $y=(y_1,y_2)\in S^1$, we see that
\begin{align*}
    \|A_{\iii}y\|^2&=|(2^{-|\iii|}y_1)^2+(s(\iii)y_1+\lambda^{|\iii|}y_2)^2| \\ 
    &= |2^{-2|\iii|}y_1^2+s(\iii)^2y_1^2+2s(\iii)\lambda^{|\iii|}y_1y_2+\lambda^{2|\iii|}y_2^2|\\
    &\leq(2^{-2|\iii|}+|s(\iii)|^2)y_1^2+2|s(\iii)|\lambda^{|\iii|}|y_1y_2|+\lambda^{2|\iii|}y_2^2 \\
    &\leq\lambda^{2|\iii|}+K_{\lambda}^2\lambda^{2|\iii|}+2K_{\lambda}\lambda^{2|\iii|}+\lambda^{2|\iii|}\\
    &=(K_{\lambda}^2+2K_{\lambda}+2)\lambda^{2|\iii|},
\end{align*}
so the claim holds with $C=\sqrt{(K_{\lambda}+1)^2+1}>1$. Finally, since
\begin{equation*}
    \alpha_2(A_{\iii})\leq \Bigl(\frac{1}{2\lambda}\Bigr)^{|\iii|}\alpha_1(A_{\iii}),
\end{equation*}
we see that $\Xl$ is dominated.
\end{proof}

Let us next determine the Furstenberg directions of the Takagi function. For a given $t \in \R$, let $V_t=\projspan{(1,t)} \in \RP$ be the line with slope $t$ passing through the origin, and let $V_{\infty}=\projspan{(0,1)} \in \RP$ be the $y$-axis. Recall also the definition of $K_\lambda$ from the formulation of Theorem \ref{thm:takagi-assouad}.

\begin{lemma}\label{lemma:furstenberg-is-interval}
  If $T_\lambda$ is the Takagi function, then
  \begin{equation*}
    X_F = \{V_t \in \RP \colon t \in [-K_\lambda,K_\lambda]\}
  \end{equation*}
  is a closed projective interval and $Y_F = \{V_{\infty}\}$ is a singleton.
\end{lemma}

\begin{proof}
Let $\iii \in \Sigma$ and observe that, by (\ref{eq:takagi-matrix}), we may define
\begin{equation*}
    B_\iii=\lim_{n\to\infty}2^{-n}A_{\overleftarrow{\iii|_n}}^{-1}=
    \begin{pmatrix}
    1&0\\
    \sum_{k=1}^{\infty}(-1)^{i_k}2^{-k}\lambda^{-k} & 0
    \end{pmatrix}.
\end{equation*}
Since $T_\lambda$ is dominated by Lemma \ref{lemma:takagi-dominated} and $\lim_{n \to \infty}\eta_1(A_{\overleftarrow{\iii|_n}}^{-1}) = (1,0)$, it follows from Lemma \ref{lemma:alternate-furstenberg-dir} that the word $\iii$ determines an element $V_{\iii} = \im(B_\iii)$ of the set $X_F$ by
\begin{align*}
    V_{\iii}=\projspan{ B_{\iii}(1,0)}=\projspan{ 1,\sum_{k=1}^{\infty}(-1)^{i_k}2^{-k}\lambda^{-k}}.
\end{align*}
Hence, it is clear that for any $V\in \{V_t \in \RP \colon t \in [-K_\lambda,K_\lambda]\}$, there exists $\iii\in\Sigma$ such that $V=V_{\iii}\in X_F$.

For the other inclusion, let $V\in X_F$. By the definition of $X_F$, we find a sequence $(c_n)_{n \in \N}$ of positive real numbers, a sequence $(\iii_n)_{n \in \N}$ of finite words $\iii_n = i^n_1 \cdots i^n_{|\iii_n|} \in \Sigma_*$, and a linear map $A$ of rank one such that $A = \lim_{n \to \infty}c_nA^{-1}_{\overleftarrow{\iii_n}}$ and $V=\im(A)$. Passing through a sub-sequence, if necessary, we see that
\begin{equation*}
    \lim_{n\to\infty}2^{-|\iii_n|}A^{-1}_{\overleftarrow{\iii_n}}=\begin{pmatrix}
    1&0\\
    \lim_{n\to\infty}\sum_{k=1}^{n}(-1)^{i^n_k}2^{-k}\lambda^{-k} & 0
    \end{pmatrix},
\end{equation*}
where the limit $\lim_{n\to\infty}\sum_{k=1}^{n}(-1)^{i^n_k+1}2^{-k}\lambda^{-k}$ exists. In particular, $c_n2^{-|\iii_n|}\to c\in \R\setminus\{0\}$ and $A$ is a constant multiple of the above matrix, which finishes the proof.

For $Y_F$, let $\iii\in\Sigma$ and using \eqref{eq:takagi-matrix} define
\begin{equation*}
    B_{\iii}=\lim_{n\to\infty} \lambda^{-n} A_{\iii|_n}=
    \begin{pmatrix}
    0&0\\
    \sum_{k=1}^{n}(-1)^{i_{n-k+1}+1}2^{-k}\lambda^{-k} & 1
    \end{pmatrix}.
\end{equation*}
Since $\im(B_{\iii})=V_{\infty}$, we have $\{V_{\infty}\}\subset Y_F$. For the other inclusion, as before, if $V\in Y_F$, we find a sequence $(c_n)_{n \in \N}$ of positive real numbers, a sequence $(\iii_n)_{n \in \N}$ of finite words $\iii_n = i^n_1 \cdots i^n_{|\iii_n|} \in \Sigma_*$, and a linear map $A$ of rank one such that $A = \lim_{n \to \infty}c_nA_{\iii_n}$ and $V=\im(A)$. Passing through a sub-sequence, if necessary, we see that
\begin{equation*}
    \lim_{n\to\infty}\lambda^{-|\iii_n|}A_{\iii_n}=\begin{pmatrix}
    0&0\\
    \sum_{k=1}^{n}(-1)^{i^n_{n-k+1}+1}2^{-k}\lambda^{-k} & 1
    \end{pmatrix},
\end{equation*}
and since $\im(A)=V_{\infty}$, the proof is finished.
\end{proof}

To finish verifying the assumptions of Theorem \ref{thm:dimasosc} for the Takagi function, it suffices to show that the Takagi function satisfies the weak bounded neighborhood condition. This is the purpose of the following lemma.

\begin{lemma}\label{lemma:takagi-sosc-bnc}
  The Takagi function $\Xl$ satisfies the BNC.
\end{lemma}

\begin{proof}
Let $C > 1$ be the constant of Lemma \ref{lemma:takagi-dominated}. Fix $x\in \Xl$ and $0<r<C^{-1}$, and let $k\in\N$ be the smallest natural number satisfying
\begin{equation*}
    2^kC^{-1}\geq 1.
\end{equation*}
Define
\begin{equation*}
    \Phi_n(x,r)=\{\varphi_{\iii}\colon 2^nC^{-1}2^{-|\iii|}\leq r < 2^nC^{-1}2^{-|\iii|+1} \text{ and } \varphi_{\iii}(X)\cap B(x,r)\ne\emptyset\}
\end{equation*}
for all $n \in \N$. It follows from Lemma \ref{lemma:takagi-dominated} that $\Phi(x,r)\subset \bigcup_{n=0}^k\Phi_n(x,r)$, so it suffices to show that the cardinality of $\Phi_n(x,r)$ is uniformly bounded for all $n \in \{0,\ldots,k\}$. For each $n\in\{0,\ldots,k\}$, let $m_n$ be the unique integer satisfying
\begin{equation}\label{eq:r}
    2^nC^{-1}2^{-m_n}\leq r < 2^{n}C^{-1}2^{-m_n+1}.
\end{equation}
Then clearly $\Phi_n(x,r)=\{\varphi_{\iii}\colon \iii\in\Sigma_{m_n}$ and $\varphi_{\iii}(X)\cap B(x,r)\ne\emptyset\}$. We observe from the construction that each $\varphi_{\iii}$ with $\iii\in\Sigma_{m_n}$ maps $T_\lambda$ inside a unique set of the form $[\frac{k}{2^{m_n}},\frac{k+1}{2^{m_n}}]\times \R$, where $k$ is an integer satisfying $0\leq k < 2^{m_n}$. Therefore, \eqref{eq:r} implies that $B(x,r)$ can intersect at most
\begin{equation*}
    2^{m_n}r\leq 2^{m_n}2^nC^{-1}2^{-m_n+1}\leq C^{-1}2^{k+1},
\end{equation*}
of the sets $\varphi_{\iii}(X)$ with $\iii\in\Sigma_{m_n}$. Since the upper bound is independent of $x$ and $r$, the set $\Xl$ satisfies the BNC.
\end{proof}

We are now ready to prove Theorem \ref{prop:takagi-assouad-formula}.

\begin{proof}[Proof of Theorem \ref{prop:takagi-assouad-formula}]
  It follows from Lemmas \ref{lemma:takagi-dominated}, \ref{lemma:furstenberg-is-interval}, and \ref{lemma:takagi-sosc-bnc} that the Takagi function $T_\lambda$ is a dominated self-affine set satisfying the BNC such that $X_F$ is a non-trivial projective interval. Furthermore, since the Takagi function is continuous, $T_\lambda(0)=0=T_\lambda(1)$ and, by \cite{MishuraSchied2019} (see also Lemma \ref{lemma:mishura-schied}), $M_\lambda = \max_{x \in [0,1]}T_\lambda(X) = \frac{1}{3(1-\lambda)} > 0$, we see that $\dimh(\proj_{V^{\perp}}(T_\lambda))=1$ for all $V\in\RP$. Therefore, by Theorem \ref{thm:dimasosc}, we have
  \begin{align*}
    \dima(\Xl)&=1+\max_{\atop{x\in T_\lambda}{V \in X_F}}\dimh(\Xl\cap (V+x))\\
    &=1+\max_{\atop{x\in T_\lambda}{V\in\RP\setminus Y_F}}\dima(\Xl\cap (V+x)).
  \end{align*}
  Furthermore, by Lemma \ref{lemma:furstenberg-is-interval}, $Y_F$ is a singleton containing only the $y$-axis $V_{\infty}$. Since $\Xl$ is a graph of a function, we have $\dima(\Xl\cap (V_{\infty}+x))=\dima(\{x\})=0$ for all $x\in \Xl$, which concludes the proof.
\end{proof}

To finish this section, we prove Theorem \ref{thm:takagi-assouad}.

\begin{proof}[Proof of Theorem \ref{thm:takagi-assouad}]
Since $T_\lambda$ is a graph of a function, it suffices to show that
\begin{equation} \label{eq:takagi-assouad-goal}
    \dimh(\Xl\cap(V+x))\leq \frac{ \log(2^{n_{\lambda}}-1)}{\log 2^{n_{\lambda}}}
\end{equation}
for all $x \in T_\lambda$ and $V \in \RP\setminus\{V_\infty\}$. We write
\begin{equation*}
    \Sigma_n(A) = \{\iii\in\Sigma_n\colon \varphi_{\iii}(\Xl)\cap A\ne\emptyset\}
\end{equation*}
for all $A \subset \R^2$ and $n \in \N$. Let us first show by induction that
\begin{equation} \label{eq:ind-sigma-est}
    \#\Sigma_{kn_{\lambda}}(V+x)\leq (2^{n_{\lambda}}-1)^k,
\end{equation}
for all $k\in\N$, $x\in T_\lambda$ and $V\in \RP\setminus\{V_\infty\}$.

First let $k=1$, $x\in T_\lambda$ and $V\in \RP\setminus\{0\}$. By symmetry, we may assume that $V=V_t$, with $t\ge0$ and without loss of generality, we may also assume that $x=(0,y)$ for some $y\in\R$ so that
\begin{equation*}
    V_t+x=\{(s,ts+y)\colon s\in\R\}.
\end{equation*}
Since $\#\Sigma_{n_{\lambda}}=2^{n_{\lambda}}$, it suffices to show that there exists at least one $\iii\in\Sigma_{n_{\lambda}}$ such that $\varphi_{\iii}(\Xl)\cap (V_t+x)=\emptyset$. Assume without loss of generality, that $\varphi_{\iii_1}(\Xl)\cap (V_t+x)\ne\emptyset$, where $\iii_1=\overline{1}|_{n_{\lambda}}$. From (\ref{eq:takagi-matrix}), we may deduce that $\varphi_{\iii_1}(\Xl)$ is contained in the rectangle $ [0,2^{-n_{\lambda}}]\times [0,(K_{\lambda}+M_{\lambda})\lambda^{n_{\lambda}}]$, and since $\varphi_{\iii_1}(\Xl)\cap (V_t+x)\ne\emptyset$, we have $y\leq (K_\lambda+M_{\lambda})\lambda^{n_{\lambda}}$. Write $x_{\frac{1}{2}}=(\frac{1}{2},\frac{1}{2})$ and let $\iii=1\overline{2}$. Since $T_{\lambda}(\frac{1}{2})=\frac{1}{2}$, we have $x_{\frac{1}{2}}\in \Xl$. In fact,
\begin{equation*}
    \varphi_{\iii|_{n_{\lambda}}}(1)=(\tfrac{1}{2},\tfrac{1}{2}).
\end{equation*}
It is also a simple exercise to show that $\varphi_{\iii|_{n_{\lambda}}}(0)\geq \frac{1}{2}$. Therefore, if $\tfrac{1}{2}t+y< \frac{1}{2}$, then $\varphi_{\iii|_{n_{\lambda}}}(X)\cap (V_t+x)=\emptyset$. On the other hand, if $t\tfrac{1}{2}+y\geq \frac{1}{2}$, then 
\begin{equation*}
    t\geq 1-2y\geq 1-2(K_{\lambda}+M_{\lambda})\lambda^{n_{\lambda}}>0
\end{equation*}
by the choice of $n_{\lambda}$. This means that the line $V_t+x$ has a positive slope. Since $\varphi_{\iii_2}(\Xl)\subset [1-2^{-n_{\lambda}},1]\times [0,(K_{\lambda}+M_{\lambda})\lambda^{n_{\lambda}}]$ by symmetry, where $\iii_2=\overline{2}|_{n_{\lambda}}$, and since $t(1-2^{n_{\lambda}})+y>\frac{1}{2}\geq (K_{\lambda}+M_{\lambda})\lambda^{n_{\lambda}}$, we have $\varphi_{\iii_2}(\Xl)\cap (V_t+x)=\emptyset$ which finishes the proof for $k=1$.

Let us then assume that \eqref{eq:ind-sigma-est} holds for $k\in\N$. Let $x\in T_\lambda$ and $V\in\RP\setminus\{0\}$. To finish the proof of \eqref{eq:ind-sigma-est}, we have to show that $\#\Sigma_{(k+1)n_{\lambda}}(V+x)\leq (2^{n_{\lambda}}-1)^{k+1}$. Notice that trivially
\begin{equation}\label{eq:cylinder-count}
    \Sigma_{(k+1)n_{\lambda}}=\{\iii\jjj\in\Sigma_{(k+1)n_{\lambda}}\colon \iii\in\Sigma_{kn_{\lambda}} \text{ and } \jjj\in\Sigma_{n_{\lambda}}\}.
\end{equation}
Let $\iii\in\Sigma_{kn_{\lambda}}$. If $\iii\not\in\Sigma_{kn_{\lambda}}(V+x)$, then $\iii\jjj\not\in\Sigma_{(k+1)n_{\lambda}}(V+x)$ for all $\jjj\in\Sigma_{n_{\lambda}}$, so we may assume that $\iii\in\Sigma_{kn_{\lambda}}(V+x)$. Since $\varphi_{\iii}$ is a bijection, we have for any $\jjj\in\Sigma_{n_{\lambda}}$, that $\varphi_{\iii\jjj}(\Xl)\cap (V+x)\ne\emptyset$ if and only if $\varphi_{\jjj}(\Xl\cap\varphi_{\iii}^{-1}(V+x))\ne\emptyset$. Since $\varphi_{\iii}^{-1}$ is affine, there exists $x_{\iii}\in T_\lambda$ and $V_{\iii}\in\RP\setminus\{V_\infty\}$ such that $\varphi_{\iii}^{-1}(V+x)=V_{\iii}+x_{\iii}$, and thus $\varphi_{\iii\jjj}(\Xl)\cap (V+x)\ne\emptyset$ if and only if $\jjj\in \Sigma_{n_{\lambda}}(V_{\iii}+x_{\iii})$. Since this is true for all $\iii\in\Sigma_{kn_{\lambda}}$, we have, by (\ref{eq:cylinder-count}) and the fact that the claim holds for $\Sigma_{n_{\lambda}}(V_\iii+x_\iii)$ and $\Sigma_{kn_{\lambda}}(V+x)$, that
\begin{align*}
    \#\Sigma_{(k+1)n_{\lambda}}(V+x)&=\sum_{\iii\in\Sigma_{kn_{\lambda}}(V+x)}\#\Sigma_{n_{\lambda}}(V_{\iii}+x_{\iii}) \\ 
    &\leq\#\Sigma_{kn_{\lambda}}(V+x)\cdot (2^{n_{\lambda}}-1) \leq (2^{n_{\lambda}}-1)^{k+1}.
\end{align*}
This concludes the proof of \eqref{eq:ind-sigma-est}.

Let us then show \eqref{eq:takagi-assouad-goal} by relying on \eqref{eq:ind-sigma-est}. It follows from the construction that for any $k\in\N$ and $\iii\in\Sigma_{kn_{\lambda}}$, the image $\varphi_{\iii}(\Xl)$ is contained in a vertical strip of width $2^{-kn_{\lambda}}$. For any $V=V_t\in\RP\setminus\{V_{\infty}\}$, we therefore have
\begin{equation} \label{eq:diam-est-sqrt}
    \diam(\varphi_{\iii}(\Xl)\cap (V_t+x))\leq 2^{-kn_{\lambda}}\sqrt{t^2+1}.
\end{equation}
Thus $\{\varphi_{\iii}(\Xl)\cap (V_t+x)\}_{\iii\in\Sigma_{kn_{\lambda}}(V_t+x)}$ is a $2^{-kn_{\lambda}}\sqrt{t^2+1}$-cover of $\Xl\cap (V_t+x)$ which together with \eqref{eq:ind-sigma-est} shows that
\begin{equation*}
    N_{2^{-kn_{\lambda}}\sqrt{t^2+1}}(\Xl\cap (V_t+x))\leq\#\Sigma_{kn_{\lambda}}(V_t+x)\leq (2^{n_{\lambda}}-1)^k.
\end{equation*}
Taking logarithms, dividing by $-\log(2^{-kn_{\lambda}}\sqrt{t^2+1})$, and letting $k\to\infty$ yields
\begin{equation*}
    \dimh(\Xl\cap (V_t+x))\leq\ldimm(\Xl\cap (V_t+x))\leq \frac{\log (2^{n_{\lambda}}-1)}{\log 2^{n_{\lambda}}}
\end{equation*}
as claimed.
\end{proof}

\section{Dimension conservation} \label{sec:proof-main2}

Let $\mu$ be the uniform Bernoulli measure on $\Sigma = \{1,2\}^\N$, that is, $\mu$ is the unique Borel probability measure with the property that $\mu([\iii]) = 2^{-|\iii|}$ for all $\iii \in \Sigma_*$. Let $\pi \colon \Sigma \to T_\lambda$ be the canonical projection onto the Takagi function. It is evident that $\nu = \pi_*\mu$ is the length measure on the $x$-axis lifted to the Takagi function. The following result, which by \eqref{eq:takagi-hausdorff-dim} and Lemma \ref{lemma:furstenberg-is-interval} is a restatement of Theorem \ref{thm:main2}, is the main result of this section.

\begin{theorem}\label{thm:equiv-marstrand}
  Suppose that $T_\lambda$ is the Takagi function and $\nu = \pi_*\mu$ is the canonical projection of the uniform Bernoulli measure. If
  \begin{equation*}
      \udimloc((\proj_{V^{\perp}})_*\nu,\proj_{V^{\perp}}(x))\geq 1
  \end{equation*}
  for all $x \in T_\lambda$ and $V\in X_F$, then
  \begin{equation} \label{eq:this-is-a-name}
      \max_{\atop{x\in T_\lambda}{V\in\RP}}\dimh(T_\lambda \cap (V+x))= \dimh(T_\lambda)-1.
  \end{equation}
  Conversely, if \eqref{eq:this-is-a-name} holds, then
  \begin{equation*}
      \ldimloc((\proj_{V^{\perp}})_*\nu,\proj_{V^{\perp}}(x))\geq 1
  \end{equation*}
  for all $x \in T_\lambda$ and $V\in\RP\setminus\{V_{\infty}\}$.
\end{theorem}

This section is devoted to the proof of this theorem. We start with an auxiliary lemma whose proof is standard. Recall that
\begin{equation*}
    \Sigma_n(A) = \{\iii\in\Sigma_n\colon \varphi_{\iii}(\Xl)\cap A\ne\emptyset\}
\end{equation*}
for all $A \subset \R^2$ and $n \in \N$.

\begin{lemma} \label{thm:minkowski-sigma}
For every $x\in \Xl$ and $V \in X_F$, we have
\begin{equation*}
    \ldimm(\Xl\cap (V+x)) = \liminf_{n\to\infty}\frac{\log\#\Sigma_n(V+x)}{n\log 2}
\end{equation*}
and 
\begin{equation*}
    \udimm(\Xl\cap (V+x)) = \limsup_{n\to\infty}\frac{\log\#\Sigma_n(V+x)}{n\log 2}.
\end{equation*}
\end{lemma}

\begin{proof}
Let $x \in T_\lambda$ and $V \in X_F$. By Lemma \ref{lemma:furstenberg-is-interval}, there is $t \in [-K_\lambda,K_\lambda]$ such that $V = V_t = \projspan{(1,t)}$. Similarly as in \eqref{eq:diam-est-sqrt}, we have
\begin{equation*}
    \diam(\varphi_{\iii}(\Xl)\cap (V+x)) \leq 2^{-n}\sqrt{t^2+1}
\end{equation*}
for all $\iii\in\Sigma_{n}(V+x)$. Therefore, the collection $\{\varphi_{\iii}(\Xl)\cap (V+x)\}_{\iii\in\Sigma_n(V+x)}$ is a $2^{-n}\sqrt{t^2+1}$-cover of $\Xl\cap (V+x)$, which proves the upper bounds. The lower bounds follow by observing that if $\{U_i\}$ is any $2^{-n}$-cover of $\Xl\cap (V+x)$, then every $U_i$ intersects at most two of the sets in $\{\varphi_{\iii}(\Xl)\cap (V+x)\}_{\iii\in\Sigma_n(V+x)}$.
\end{proof}

The above lemma connects $\Sigma_n(V+x)$ to the Minkowski dimensions of the slices. As we further wish to connect the pointwise dimensions on $T_\lambda$ to the slices, we are interested in estimating the number of words $\iii$ not in $\Sigma_n(V+x)$ for which $\fii_\iii(T_\lambda)$ is still relatively close to $V+x$. The \emph{$r$-neighborhood} of a set $A \subset \R^2$ is denoted by $[A]_r = \{x \in \R^2 : |x-y| \le r$ for some $y \in A\}$. For $n\in\N$ and $c>0$, we define the \emph{set of bad words at level $n$} by
\begin{equation*}
    \Bad_{n,c} = \Sigma_n([V+x]_{c\lambda^n}) \setminus \Sigma_n(V+x).
\end{equation*}
We say that a bad word $\iii$ at level $n$ is \emph{generated at level $k$} if $k \in \{1,\ldots,n\}$ is the smallest number such that $\iii|_{k}\not\in\Sigma_{k}(V+x)$, and we denote the set of these length $n$ words by $\Bad_{n,c}^k$. Since every bad word at level $n$ is generated at exactly one level $k \le n$, we have
\begin{equation*}
    \#\Bad_{n,c}=\sum_{k=1}^{n}\#\Bad_{n,c}^k.
\end{equation*}
The following lemma is the key observation in our analysis.

\begin{lemma}\label{lemma:few-bad}
  For every $x\in \Xl$ and $V\in \RP\setminus\{V_{\infty}\}$ there are constants $c,K>0$ such that
  \begin{equation*}
      \#\Bad_{n,c}^k\leq K\cdot\#\Sigma_k(V+x)
  \end{equation*}
  for all $n,k \in \N$ with $k \le n$.
\end{lemma}

The proof of the lemma is technical and takes several pages. Trying not to disrupt the flow of the presentation, we have postponed it into Subsection \ref{subsec:crux}. Lemma \ref{lemma:few-bad} allows us to connect the pointwise dimensions of the length measure on the $x$-axis lifted to the Takagi function with the Minkowski dimensions of the slices. This is the content of the following proposition. It generalizes a similar result of Manning and Simon \cite[Proposition 4]{ManningSimon2013} for the Sierpi\'nski carpet to the self-affine regime.

\begin{proposition}\label{prop:weak-dimension-conservation}
  If $\Xl$ is the Takagi function and $\nu = \pi_*\mu$ is the canonical projection of the uniform Bernoulli measure, then
  \begin{equation*}
      \udimloc((\proj_{V^{\perp}})_*\nu, \proj_{V^{\perp}}(x)) + \frac{\log \frac12}{\log\lambda}\ldimm(\Xl\cap (V+x)) = \frac{\log \frac12}{\log\lambda}
  \end{equation*}
  and 
  \begin{equation*}
      \ldimloc((\proj_{V^{\perp}})_*\nu, \proj_{V^{\perp}}(x)) + \frac{\log \frac12}{\log\lambda}\udimm(\Xl\cap (V+x)) = \frac{\log \frac12}{\log\lambda}
  \end{equation*}
  for all $x\in \Xl$ and $V\in\RP\setminus\{V_{\infty}\}$.
\end{proposition}

\begin{proof}
Let $x\in \Xl$ and $V\in\RP\setminus\{V_{\infty}\}$, and note that
\begin{equation*}
    (\proj_{V^{\perp}})_*\nu(B(\proj_{V^{\perp}}(x),r))=\nu([V+x]_r)
\end{equation*}
for all $r > 0$. Write $c = \sqrt{2}(K_{\lambda}+M_{\lambda})$. It is easy to see that for every $\iii\in\Sigma_n$ we have
\begin{equation*}
    \diam(\varphi_{\iii}(\Xl))\leq c\lambda^n,
\end{equation*}
so for any $\iii\in\Sigma_n(V+x)$, the cylinder $\varphi_{\iii}(\Xl)$ is contained in $[V+x]_{c\lambda^{n}}$. Since the $\nu$-measure of $\varphi_{\iii}(\Xl)$ is $2^{-n}$, we have
\begin{equation*}
    \nu([V+x]_{c\lambda^{n}}) \geq \sum_{\iii\in\Sigma_n(V+x)} \nu(\varphi_{\iii}(\Xl))\geq 2^{-n}\cdot\#\Sigma_n(V+x).
\end{equation*}
Therefore
\begin{align*}
    \frac{\log\nu([V+x]_{c\lambda^{n}})}{\log c\lambda^n}&\leq\frac{-n\log2}{\log c\lambda^n}+\frac{\log\#\Sigma_n(V+x)}{\log c\lambda^n}\\
    &=\frac{n\log\frac12}{\log c\lambda^n}-\frac{n\log \tfrac12}{\log c\lambda^n}\frac{\log\#\Sigma_{n}(V+x)}{n\log2},
\end{align*}
and taking the limit superior or the limit inferior, Lemma \ref{thm:minkowski-sigma} yields the respective upper bounds.

The lower bound is more subtle and for that we apply Lemma \ref{lemma:few-bad}. Let $c,K>0$ be as in Lemma \ref{lemma:few-bad}. Since the collection $\{\varphi_{\iii}(X)\colon \iii\in\Sigma_n(V+x)\cup\Bad_{n,c}\}$ covers $T_\lambda \cap[V+x]_{c\lambda^n}$, Lemma \ref{lemma:few-bad} shows that
\begin{align*}
    \nu([V+x]_{c\lambda^n})&\leq\sum_{\iii\in\Sigma_n(V+x)}\nu(\varphi_{\iii}(T_\lambda))+\sum_{\iii\in\Bad_{n,c}}\nu(\varphi_{\iii}(T_\lambda))\\
    &=2^{-n}\cdot\#\Sigma_n(V+x)+2^{-n}\cdot\#\Bad_{n,c}\\
    &=2^{-n}\cdot\#\Sigma_n(V+x)+2^{-n}\sum_{k=1}^{n}\#\Bad_{n,c}^k\\
    &\leq2^{-n}\cdot\#\Sigma_n(V+x)+2^{-n}K\sum_{k=1}^{n}\#\Sigma_k(V+x)\\
    &\leq2^{-n}(Kn+1)\cdot\#\Sigma_n(V+x).
\end{align*}
Taking logarithms, dividing by $c\lambda^n$, and taking the limits then gives the desired lower bounds.
\end{proof}

We are now ready to prove Theorem \ref{thm:equiv-marstrand}.

\begin{proof}[Proof of Theorem \ref{thm:equiv-marstrand}]
Let us first assume that
\begin{equation*}
    \udimloc((\proj_{V^{\perp}})_*\nu,\proj_{V^{\perp}}(x))\geq 1
\end{equation*}
for all $x \in T_\lambda$ and $V\in X_F$. Then Proposition \ref{prop:weak-dimension-conservation} and \eqref{eq:takagi-hausdorff-dim} give us
\begin{align*}
    \dimh(\Xl\cap(V+x))&\leq\ldimm(\Xl\cap(V+x))\\
    &=1-\frac{\log\lambda}{\log\frac12}\udimloc((\proj_{V^{\perp}})_*\nu,\proj_{V^{\perp}}(x))\\
    &\leq 1+\frac{\log\lambda}{\log 2}=\dimh(\Xl)-1,
\end{align*}
for all $x\in \Xl$ and $V \in X_F$. Since $T_\lambda$ is a graph of a function and recalling the proof of Theorem \ref{thm:dimasosc}, this estimate extends to all $V \in \RP$. Therefore, by Theorem \ref{prop:takagi-assouad-formula},
\begin{equation*}
    \dimh(\Xl)-1\leq \dima(\Xl)-1=\max_{\atop{x\in T_\lambda}{V\in \RP}}\dimh(T_\lambda \cap (V+x))\leq \dimh(\Xl)-1
\end{equation*}
and the claim follows.

Let us then assume that
\begin{equation} \label{eq:equiv-marstrand1}
    \max_{\atop{x\in T_\lambda}{V\in\RP}}\dimh(T_\lambda \cap (V+x))= \dimh(T_\lambda)-1.
\end{equation}
If $x\in T_\lambda$ and $V\in \RP\setminus\{V_{\infty}\}$ are such that
\begin{equation*}
  \ldimloc((\proj_{V^{\perp}})_*\nu,\proj_{V^{\perp}}(x))<1,
\end{equation*}
then, by Proposition \ref{prop:weak-dimension-conservation}, \eqref{eq:takagi-hausdorff-dim}, \eqref{eq:equiv-marstrand1}, and Theorem \ref{prop:takagi-assouad-formula}, we have
\begin{align*}
    \udimm(\Xl\cap(V+x))&>1-\frac{\log\lambda}{\log\frac12}=1+\frac{\log\lambda}{\log 2}=\dimh(T_\lambda)-1\\ 
    &=\max_{\atop{x\in T_\lambda}{V\in\RP}}\dimh(T_\lambda \cap (V+x))=\max_{\atop{x\in T_\lambda}{V\in\RP}}\dima(T_\lambda \cap (V+x))
\end{align*}
which is a contradiction.
\end{proof}

\subsection{Proof of Lemma \ref{lemma:few-bad}}\label{subsec:crux}
It remains to prove Lemma \ref{lemma:few-bad}. The proof we give is quite technical, but the tools are elementary. The following geometric lemma allows us to simplify the problem. Write $\RP_{\Theta}= \{\projspan{(1,t)}\in\RP\colon |t|\leq \Theta\}$ for all $\Theta>0$ and let $K_\lambda > 0$ be as in Theorem \ref{thm:takagi-assouad}.

\begin{lemma}\label{lemma:pull-back-cylinder}
    For any $\Theta\geq K_{\lambda}$ there is a constant $C=C(\Theta)>0$ such that for every $n\in\N$, $\iii = i_1 \cdots i_n \in \Sigma_n$, $r>0$, $x\in \R^2$, and $V_t\in \RP_{\Theta}$ there exists $y\in\R^2$ such that
    \begin{equation*}
         \varphi_{\iii}^{-1}([V_t+x]_{r})\subset [V_{t_{\iii}}+y]_{C\lambda^{-n}r},
    \end{equation*}
    where $t_{\iii}=\sum_{k=1}^{n}(-1)^{i_k}2^{-k}\lambda^{-k}+(2\lambda)^{-n}t$. Furthermore, we have $V_{t_{\iii}}\in \RP_{\Theta}$.
\end{lemma}

\begin{proof}
It follows from (\ref{eq:takagi-matrix}), that $\varphi_{\iii}^{-1}$ maps $V_t+x$ to a line $V_{t_{\iii}}+y$ with slope $t_{\iii}=\sum_{k=1}^{n}(-1)^{i_k}2^{-k}\lambda^{-k}+(2\lambda)^{-n}t$, and since $\varphi_{\iii}^{-1}$ expands vertical distances by $\lambda^{-n}$, a simple geometric argument shows that by taking $C=\sqrt{\Theta^2+1}$, we have
\begin{equation*}
    \varphi_{\iii}^{-1}([V_t+x]_{r}) = [V_{t_{\iii}}+y]_{\frac{\sqrt{t^2+1}}{\sqrt{t_{\iii}^2+1}}\lambda^{-n}r}\subset [V_{t_{\iii}}+y]_{C\lambda^{-n}r}.
\end{equation*}
If $|t|\leq K_{\lambda}$, then $V_{t_{\iii}}\in X_F$ by Lemma \ref{lemma:rossi} and therefore, $|t_{\iii}|\leq K_{\lambda}\leq \Theta$ by Lemma \ref{lemma:furstenberg-is-interval}. Furthermore, if $\frac{1}{2\lambda-1}=K_{\lambda}<|t|\le\Theta$, then
\begin{align*}
    |t_{\iii}|&\leq \sum_{k=1}^n(2\lambda)^{-k}+(2\lambda)^{-n}|t|=\frac{1-(2\lambda)^{-n}}{2\lambda-1}+(2\lambda)^{-n}|t|\\
    &=(1-(2\lambda)^{-n})K_{\lambda}+(2\lambda)^{-n}|t|\leq |t|\leq \Theta.
\end{align*}
Therefore, $V_{t_{\iii}}\in \RP_{\Theta}$.
\end{proof}

Fix $x\in \Xl$ and $V\in\RP\setminus\{V_{\infty}\}$. Note that $V\in \RP_{\Theta}$ for some $\Theta\geq K_{\lambda}$. Define
\begin{equation}\label{eq:bad-children}
    \Bad^k_{n,c}(\iii)=\{\iii\jjj\in\Sigma_n\colon\jjj\in\Sigma_{n-k} \text{ and } \varphi_{\iii\jjj}(\Xl)\cap [V+x]_{c\lambda^n}\ne\emptyset\}
\end{equation}
for all $\iii\in\Sigma_k$ and $k \le n$. To prove Lemma \ref{lemma:few-bad} it suffices to show that there are constants $c,K>0$ such that $\#\Bad^k_{n,c}(\iii)\leq K$ for all $k \le n$, {since in this case, we have
\begin{equation*}
   \#\Bad_{n,c}^k=\sum_{\iii\in\Bad_{k,c}^k}\#\Bad^k_{n,c}(\iii)\leq K\#\Bad^k_{k,c}\leq K\cdot\#\Sigma_{k-1}(V+x).
\end{equation*}}
Moreover, by (\ref{eq:bad-children}) and Lemma \ref{lemma:pull-back-cylinder}, we have
\begin{equation*}
    \Bad^k_{n,c}(\iii)\subset\{\iii\jjj\in\Sigma_n\colon \jjj\in \Sigma_{n-k}([V_{t_{\iii}}+y]_{Cc\lambda^{n-k}})\},
\end{equation*}
so in particular, $\#\Bad^k_{n,c}(\iii)\leq\#\Sigma_{n-k}([V_{t_{\iii}}+y]_{Cc\lambda^{n-k}})$, where $V_{t_{\iii}}\in \RP_{\Theta}$. Note also that if $\iii\in\Bad_{k,c}^k$, then we have $\Xl\cap (V_{t_{\iii}}+y)=\emptyset$. Therefore, the following lemma implies Lemma \ref{lemma:few-bad}.

\begin{lemma}\label{lemma:cylinder-bound}
    For any $\Theta\geq K_{\lambda}$ there are constants $C=C(\Theta)>0$ and $K=K(\Theta)\in\N$ such that for every $V\in \RP_{\Theta}$ and $x\in\R^2$ satisfying $\Xl\cap (V+x)=\emptyset$, we have
    \begin{equation*}
        \#\Sigma_n([V+x]_{C\lambda^n})\leq K
    \end{equation*}
    for all $n\in\N$.
\end{lemma}

The remainder of the paper is dedicated to the proof of Lemma \ref{lemma:cylinder-bound}. We first present six geometric lemmas which further clarify the situation and then conclude the proof at the end of the section. Let us recall the following result of Mishura and Schied \cite{MishuraSchied2019}.

\begin{lemma}\label{lemma:mishura-schied}
    For $\tfrac12<\lambda<1$, the Takagi function $T_{\lambda}$ has exactly two maximizers, $x_{\max}=\frac{1}{3}$ and $y_{\max}=\frac{2}{3}$, and its maximum value is $M_{\lambda}=\frac{1}{3(1-\lambda)}$. Furthermore, $\iii_{\max}=\overline{12}$ and $\jjj_{\max}=\overline{21}$ are the only infinite words with $\pi(\iii_{\max}) = (x_{\max},M_\lambda)$ and $\pi(\jjj_{\max}) = (y_{\max},M_\lambda)$.
\end{lemma}

\begin{proof}
    The first part of the claim is proved in \cite{MishuraSchied2019}. The fact that each maximum in $T_\lambda$ has a unique coding follows from the fact that the projection of the IFS $\{\varphi_1,\varphi_2\}$ onto the $x$-axis is the IFS which generates the dyadic intervals, and $\frac{1}{3}$ and $\frac{2}{3}$ have unique dyadic codings $\iii_{\max}=\overline{12}$ and $\jjj_{\max}=\overline{21}$, respectively.
\end{proof}

In the following, we rely heavily on the mirror symmetry of $\Xl$. By Lemma \ref{lemma:mishura-schied}, the Takagi function restricted to $[0,\tfrac12]$, $\varphi_{1}(\Xl)$, has a unique maximum at $x_{\max}$. We denote the point on the graph of $T_\lambda$ corresponding to this maximum by $\overline{x}_{\max}$, that is $\overline{x}_{\max}=\pi(\iii_{\max}) = (x_{\max},M_\lambda)$.  We write $\iii_1=\overline{1}$ and $\iii_2=\overline{2}$. Let $\Sigma^1_n=\{\iii=i_1\cdots i_n\in\Sigma_n\colon i_1=1\}$ and
\begin{equation*}
  \Sigma_n^1(A)=\{\iii\in\Sigma_n^1\colon \varphi_{\iii}(\Xl)\cap A\ne\emptyset\}
\end{equation*}
for all $A \subset \R^2$ and $n \in \N$. For each $x\in \R^2$, $V\in \RP\setminus\{V_{\infty}\}$, $v\in V$ with $|v|=1$, and $\delta>0$ we define
\begin{align*}
    &C(x,v,\delta)=\{y\in\R^2\colon (y-x)\cdot v < (1+\delta^2)^{-1/2}|y-x|\},\\
    &H^+(x,V)=\{y\in\R^2\colon v^{\perp}\cdot (y-x)>0\},\\
    &H^-(x,V)=\{y\in\R^2\colon v^{\perp}\cdot (y-x)<0\},
\end{align*}
where $v^{\perp}$ is the unique vector with positive second coordinate orthogonal to $v$. Further, we let $H^-(x,V_\infty)$ and $H^+(x,V_\infty)$ denote the left open half-plane and the right open half-plane centered at $x$, respectively. Finally, let $C^-(x,\delta)=C(x,(-1,0),\delta)$ and $C^+(x,\delta)=C(x,(1,0),\delta)$. 

\begin{lemma}\label{lemma:empty-cone}
    If $x=(x_{\max},y)$, where $y>M_{\lambda}$, then $\Xl\cap C^-(x,1)=\emptyset$.
\end{lemma}
    
\begin{proof}
Let $A=\R\times (-\infty,M_{\lambda}]$ and note that
\begin{equation*}
    \Xl\subset\varphi_{1}(A)=\{(x,y)\in\R^2\colon x\in\R \text{ and } y \leq x+\lambda M_{\lambda}\}.
\end{equation*}
In particular, the graph of the Takagi function lies below the line satisfying the equation $y=x+\lambda M_{\lambda}$. Note that the point $\overline{x}_{\max}$ is on this line, so $\varphi_1(T_{\lambda})$ lies below the line $V_1+\overline{x}_{\max}$, and the claim follows.
\end{proof}

\begin{figure}[h!]
    \captionsetup{justification=justified}
    \centering
    \begin{tikzpicture}[scale=0.9]
    \filldraw[color=black!0, fill=gray!20] (0,3) rectangle (-6,6);
    \draw (-6,0) -- node[anchor=north, pos=0.7] {$r/t$} (6,0) node[anchor=north east] {$V_0+x$};
    \filldraw (0,0) circle (1.5pt) node[anchor=south east] {$x$} -- (0,6);
    \draw[dashed] (-6,3) -- (6,3);
    \draw (-3,-1.5) -- (6,3) node[anchor=south east, pos=0.7] {$V_t+x$};
    \draw[dashed] (-6,-1) -- (6,5);

    \draw (6,0) -- node[anchor=west] {$r$} (6,3);
    \draw (0,2) -- node[anchor=west] {$Cr$} (0.8,0.4);
    \node[anchor=east] at (0,1) {$h$};

    \coordinate (A) at (0,0);
    \coordinate (B) at (0,2);
    \coordinate (C) at (0.8,0.4);
    \coordinate (D) at (6,0);
    \coordinate (E) at (6,3);
    \draw pic[draw] {angle};
    \draw (D) -- (A) -- (E) pic[draw] {angle = D--A--E};
    \MarkRightAngle[size=10pt](E,D,A);
    \MarkRightAngle[size=7pt](B,C,A);
    \end{tikzpicture}
    \caption{The geometric observation of Lemma \ref{lemma:geometric1}.}
    \label{fig:geometric1}
\end{figure}
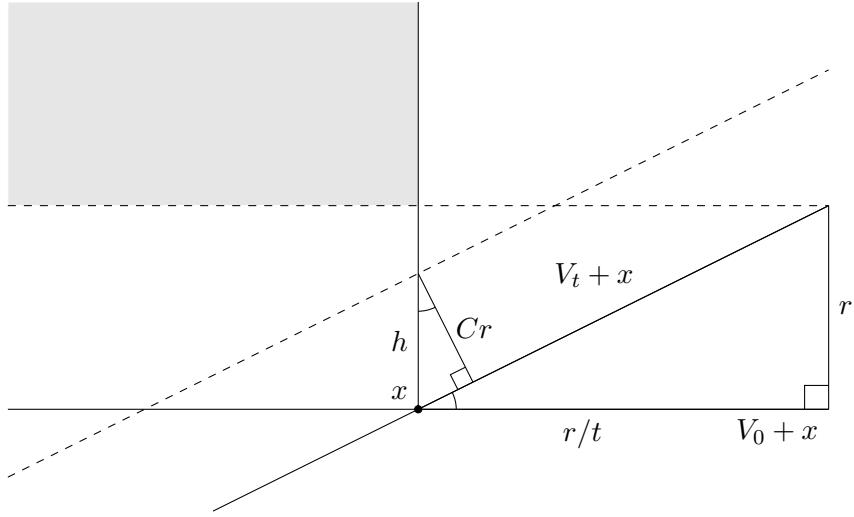

\begin{lemma}\label{lemma:geometric1}
    Let $\Theta>0$, $x\in\R^2$, $V_t\in \RP_{\Theta}$ with $t>0$, and $A\subset H^+(x,V_0)\cap H^-(x,V_{\infty})$. There exists a constant $C=C(\Theta)>0$ such that for any $r>0$ satisfying $A\cap[V_0+x]_r=\emptyset$ we have $A\cap[V_t+x]_{Cr}=\emptyset$.
\end{lemma}

\begin{proof}
Let $C= (\sqrt{\Theta^2+1})^{-1}$. Since $A$ is contained in the shaded area in Figure \ref{fig:geometric1}, it suffices to show that $h\leq r$. Since the triangles are similar, we have $h=\sqrt{t^2+1}Cr \leq \sqrt{\Theta^2+1}Cr \leq r$.
\end{proof}

\begin{lemma}\label{lemma:geometric2}
    Let $x\in\R^2$, $\delta>0$, and $A\subset H^-(x,V_0)$ be such that $A\cap C^-(x,\delta)=\emptyset$ and $A\cap [V_0+x]_r=\emptyset$. Then for any $0<\varepsilon<\delta$ there exists a constant $C=C(\varepsilon,\delta)>0$ such that for each $0<t\leq\delta-\varepsilon$ we have $A\cap[V_t+x]_{Cr}=\emptyset$.
\end{lemma}

\begin{proof}
Let $C= \frac{\varepsilon}{\delta\sqrt{\delta^2+1}}$. The claim follows from the following geometric observation. Notice from Figure \ref{fig:geometric2} that, since $A$ is contained in the shaded area, it suffices to show $d\geq \frac{r}{\delta}$. Since the right-angled triangles with side lengths $r$ and $r/t$, and $\ell$ and $Cr$ are similar, we have $\ell=C\sqrt{1+t^{-2}}r$, and therefore
\begin{equation*}
    d=\frac{r}{t}-C\sqrt{1+t^{-2}}r=\biggl(\frac{1}{t}-\frac{\varepsilon\sqrt{t^2+1}}{t\delta\sqrt{\delta^2+1}}\biggr)r\geq \biggl(\frac{\delta-\varepsilon}{t\delta}\biggr)r\geq \frac{r}{\delta}
\end{equation*}
as claimed.
\end{proof}

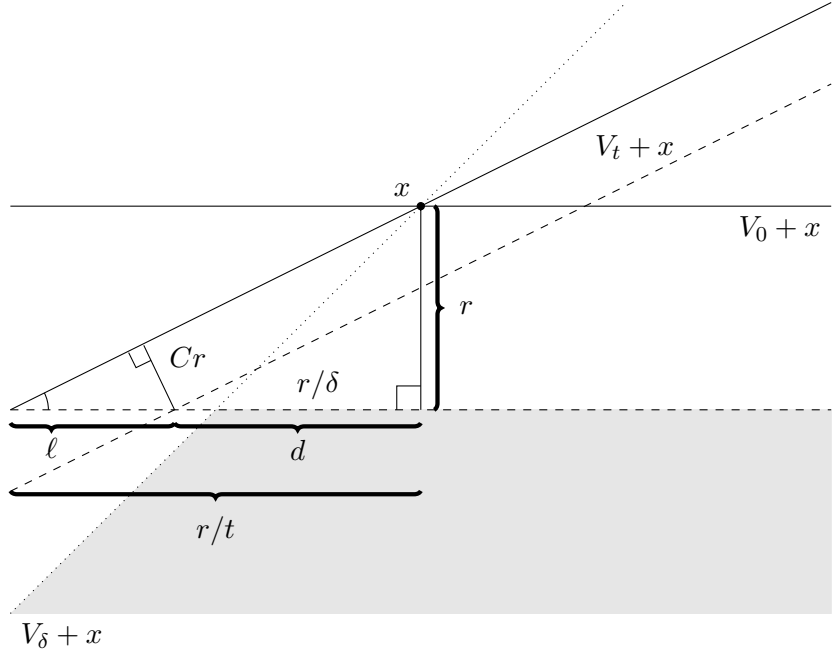
\begin{figure}
    \captionsetup{justification=justified}
    \centering
    \begin{tikzpicture}[scale=0.9]
    \filldraw[color=gray!20, fill=gray!20] (-3,-3) -- (6,-3) -- (6,-6) -- (-6,-6);

    \draw (-6,0) -- (6,0) node[anchor=north east] {$V_0+x$};
    \filldraw (0,0) circle (1.5pt) node[anchor=south east] {$x$} -- (0,-3);
    \draw[dashed] (-6,-3) -- (6,-3);
    \draw (-6,-3) -- (6,3) node[anchor=north west, pos=0.7] {$V_t+x$};
    \draw[dashed] (-6,-4.2) -- (6,1.8);
    \draw[dotted] (-6,-6) node[anchor=north west] {$V_{\delta}+x$} -- (3,3);
    \draw (-3.6,-3) -- node[anchor=south west] {$Cr$} (-4.05,-2.05);

    \node[anchor=south] at (-1.5,-3) {$r/\delta$};

    \draw[decorate, decoration={brace, raise=5pt}, ultra thick] (0,0) -- (0,-3)
    node[right=10pt,pos=0.5] {$r$};
    \draw[decorate, decoration={brace, raise=5pt, aspect=0.75}, ultra thick] (-3.6,-3) -- (-6,-3)
    node[below=6pt,pos=0.75] {$\ell$};
    \draw[decorate, decoration={brace, raise=5pt}, ultra thick] (0,-3) -- (-3.6,-3)
    node[below=7pt,pos=0.5] {$d$};
    \draw[decorate, decoration={brace, raise=5pt}, ultra thick] (0,-4) -- (-6,-4)
    node[below=10pt,pos=0.5] {$r/t$};

    \coordinate (A) at (0,0);
    \coordinate (B) at (0,-3);
    \coordinate (C) at (-6,-3);
    \coordinate (D) at (-4.05,-2.05);
    \coordinate (E) at (-3.6,-3);
    \draw pic[draw] {angle = B--C--A};
    \MarkRightAngle[size=10pt](A,B,C);
    \MarkRightAngle[size=7pt](C,D,E);
    \end{tikzpicture}
    \caption{The geometric observation of Lemma \ref{lemma:geometric2}.}
    \label{fig:geometric2}
\end{figure}

In the remaining lemmas we use the following notation. For $n\in\N$ we let $\iii^L_n=\iii_{\max}|_{2n-2}11$, $\iii^R_n=\iii_{\max}|_{2n-2}21$, and $\iii^{RR}_{n}=\iii_{\max}|_{2n-2}22$. The geometric interpretation of this is that $\varphi_{\iii^L_n}(\Xl)$ corresponds to the cylinder adjacent to $\varphi_{\iii_{\max}|_{2n}}(\Xl)$ on the left-hand side, $\varphi_{\iii^R_n}(\Xl)$ to the cylinder on the right-hand side, and $\varphi_{\iii^{RR}_n}(\Xl)$ to the cylinder adjacent to $\varphi_{\iii^R_n}(\Xl)$ on the right-hand side. In the following, for any $x\in\R^2$, we write
\begin{equation*}
    |x|_y= |\proj_{V_{\infty}}(x)|.
\end{equation*}
This is a seminorm on $\R^2$ and it becomes a norm if we identify points with equal $y$-coordinates. The seminorm $| \cdot |_y$ induces a translation invariant pseudometric in the usual way. For $x\in\R^2$ and $E\subset \R^2$ we let $\dist_y(x,E)= \inf_{z\in E}|x-z|_y$. For every $x\in\R^2$ and $r>0$, we have
\begin{equation*}
    B_y(x,r)=[V_0+x]_r,
\end{equation*}
where $B_y(x,r)$ denotes the open ball with center $x$ and radius $r$ in the pseudometric induced by $| \cdot |_y$. In the sequel, we will repeatedly use the following simple fact: If $\fii \colon \R^2 \to \R^2$, $\varphi(x)=Ax+t$, is an affine map, then
\begin{equation}\label{eq:affine-translate}
    |\varphi(x_1)-\varphi(x_2)|_y=|A(x_1-x_2)|_y.
\end{equation}
for all $x_1,x_2\in\R^2$.

\begin{lemma}\label{lemma:vertical-distance}
    There exists a constant $C=C(\lambda)>0$ such that for any integer $n>1$ and $\iii\in\{\iii_n^L,\iii_n^R,\iii_n^{RR}\}$ we have
    \begin{equation*}
        \dist_y(\overline{x}_{\max},\varphi_{\iii}(\Xl))\geq C\lambda^{2n}.
    \end{equation*}
    Moreover, if $\iii=\iii_n^L$, then the claim holds also with $n=1$.
\end{lemma}

\begin{proof}
    Let us denote $p_n= \varphi_{\iii_{\max}|_{2n}}(0,M_{\lambda})=\varphi_{\iii_{n}^L}(1,M_{\lambda})$, let $A=[0,1]\times [0,M_\lambda]$ and let $T$ be the open triangle determined by the points $(1,M_\lambda)$, $(\frac{2}{3},M_\lambda)$, and $(1,M_\lambda-\frac{1}{3})$. Note that the affine map $\varphi_{\iii^L_n}$ maps $T$ to a triangle determined by the points $\varphi_{\iii^L_n}(1,M_\lambda)=p_n$, $\varphi_{\iii^L_n}(\frac{2}{3},M_\lambda)$, and $\varphi_{\iii^L_n}(1,M_\lambda-\frac{1}{3})$. By Lemma \ref{lemma:empty-cone}, we have $\Xl\subset A\setminus T$ which gives $\varphi_{\iii^L_n}(\Xl)\subset \varphi_{\iii^L_n}(A\setminus T)$, see Figure \ref{fig:crux_fig} for illustration. {Using \eqref{eq:takagi-matrix}, the non-zero, non-diagonal element of the matrix $A_{\iii_n^L}$ is $$
\biggl((2\lambda)^{-1}+(2\lambda)^{-2}+\sum_{k=3}^{2n}(-1)^k(2\lambda)^{-k}\biggr)\lambda^{2n}=\frac{(2\lambda)^{-2n}+\lambda+1}{2\lambda^2+\lambda}\lambda^{2n}\geq\frac{\lambda+1}{2\lambda^2+\lambda}\lambda^{2n}>0,
$$
} and therefore
    \begin{equation*}
        \dist_y(p_n,\varphi_{\iii^L_n}(\Xl))\geq \min\{|p_n-\varphi_{\iii^L_n}(\tfrac{2}{3},M_\lambda)|_y,|p_n-\varphi_{\iii^L_n}(1,M_\lambda-\tfrac{1}{3})|_y\}.
    \end{equation*}
    \begin{figure}
    	\centering
    	\captionsetup{justification=justified}
    	\includegraphics[width=0.66\linewidth]{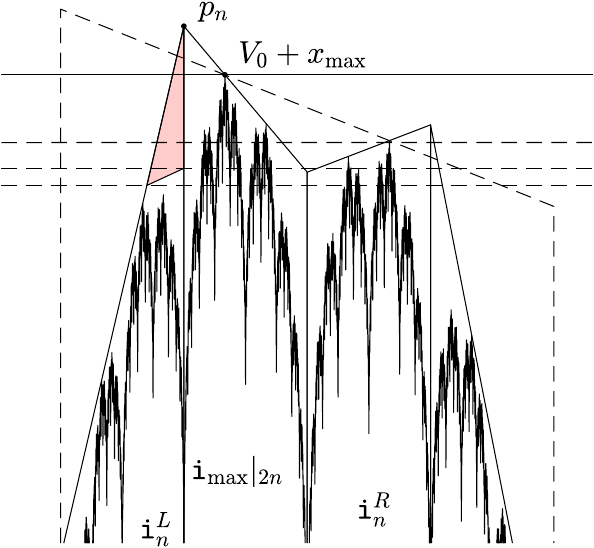}
    	\caption{Illustration of the proof of Lemma \ref{lemma:vertical-distance}.}
    	\label{fig:crux_fig}
    \end{figure}
 Using equations \eqref{eq:affine-translate} and \eqref{eq:takagi-matrix}, we have
    \begin{align*}
        |p_n-\varphi_{\iii^L_n}(1,M_\lambda-\tfrac{1}{3})|_y&=|\varphi_{\iii_{n}^L}(1,M_{\lambda})-\varphi_{\iii^L_n}(1,M_\lambda-\tfrac{1}{3})|_y\\
        &=|A_{\iii^L_n}(0,\tfrac{1}{3})|_y=\frac{\lambda^{2n}}{3}
    \end{align*}
    and, similarly,
    \begin{align*}
        |p_n-\varphi_{\iii^L_n}(\tfrac{2}{3},M_\lambda)|_y&=|\varphi_{\iii_{n}^L}(1,M_{\lambda})-\varphi_{\iii^L_n}(\tfrac{2}{3},M_\lambda)|_y=|A_{\iii^L_n}(\tfrac{1}{3},0)|_y\geq\frac{\lambda+1}{6\lambda^2+3\lambda}\lambda^{2n}.
    \end{align*}
    Therefore, we have
    \begin{equation}\label{eq:pn_to_L}
        \dist_y(p_n,\varphi_{\iii^L_n}(\Xl))\geq\min\biggl\{\frac{1}{3},\frac{\lambda+1}{6\lambda^2+3\lambda}\biggr\}\lambda^{2n}.
    \end{equation}
    On the other hand, it follows by a simple calculation that $\varphi_{12}(\overline{x}_{\max})=\overline{x}_{\max}$, and therefore, by induction, we have $\varphi_{\iii_{\max}|_{2n}}(\overline{x}_{\max})=\overline{x}_{\max}$ for all $n\in\N$. By \eqref{eq:affine-translate} and \eqref{eq:takagi-matrix}, we have
    \begin{align*}
        |p_n-\overline{x}_{\max}|_y &=|\varphi_{\iii_{\max}|_{2n}}(0,M_\lambda)-\varphi_{\iii_{\max}|_{2n}}(\tfrac{1}{3},M_\lambda)|_y=|A_{\iii_{\max}}\left(\tfrac{1}{3},0\right)|_y\\
        &=\biggl|\sum_{k=1}^{2n}(-1)^{k}(2\lambda)^{-k}\biggr|\frac{\lambda^{2n}}{3}=\frac{1-(2\lambda)^{-2n}}{3(2\lambda+1)}\lambda^{2n}\leq \frac{\lambda^{2n}}{3(2\lambda+1)}.
    \end{align*}
    By combining this with (\ref{eq:pn_to_L}) and applying the reverse triangle inequality, we get
    \begin{align*}
        \dist_y(\overline{x}_{\max},\varphi_{\iii_n^L}(\Xl))&\geq \dist_y(p_n,\varphi_{\iii_n^L}(\Xl))-|p_n-\overline{x}_{\max}|_y\\
        &\geq \biggl(\min\biggl\{\frac{1}{3},\frac{\lambda+1}{6\lambda^2+3\lambda}\biggr\}-\frac{1}{3(2\lambda+1)}\biggr)\lambda^{2n}\\
        &=\min\biggl\{\frac{2\lambda}{6\lambda+3},\frac{1}{6\lambda^2+3\lambda}\biggr\}\lambda^{2n},
    \end{align*}
    for any $n\in\N$, where $\min\{\frac{2\lambda}{6\lambda+3},\frac{1}{6\lambda^2+3\lambda}\}>0$ for all $\tfrac12 < \lambda < 1$.

    It follows from the construction, that $\dist_y(\overline{x}_{\max},\varphi_{\iii_n^R}(\Xl))\leq \dist_y(\overline{x}_{\max},\varphi_{\iii_n^{RR}}(\Xl))$ (see Figure \ref{fig:crux_fig}), so to finish the proof, it is enough to prove the claim for $\iii_n^R$. 
    {Observe that
$$
\dist_y(\overline{x}_{\max},\varphi_{\iii_n^R}(\Xl))=\min\{|\overline{x}_{\max}-\varphi_{\iii_n^R}(\overline{x}_{\max})|_y,|\overline{x}_{\max},\varphi_{\iii_n^R}(\overline{y}_{\max})|_y\},
$$
where
\begin{align*}
|\overline{x}_{\max}-\varphi_{\iii_n^R}(\overline{x}_{\max})|_y&=|\varphi_{\iii_{\max}|_{2n-2}}(\overline{x}_{\max})-\varphi_{\iii_n^R}(\overline{x}_{\max})|_y\\
&=|A_{\iii_{\max}|_{2n-2}}(\overline{x}_{\max}-\varphi_{21}(\overline{x}_{\max}))|_y\\
&=\left|A_{\iii_{\max}|_{2n-2}}(\tfrac14,(1-\lambda^2)M_\lambda+\tfrac{1}{12}-\tfrac{\lambda}{6})\right|_y\\
&=\biggl|\frac14\sum_{k=1}^{2n-2}(-1)^{k}(2\lambda)^{-k}+(1-\lambda^2)M_\lambda+\tfrac{1}{12}-\tfrac{\lambda}{6}\biggr|\lambda^{2n-2}\\
&\geq(-\tfrac{1}{8\lambda}+\tfrac{1+\lambda}{3}+\tfrac{1}{12}-\tfrac{\lambda}{6})\lambda^{2n-2}
\end{align*}
and, similarly,   
\begin{align*}
|\overline{x}_{\max}-\varphi_{\iii_n^R}(\overline{y}_{\max})|_y&=\left|A_{\iii_{\max}|_{2n-2}}(\tfrac16,(1-\lambda^2)M_\lambda+\tfrac{1}{6}-\tfrac{\lambda}{3})\right|_y\\
&=\biggl|\frac16\sum_{k=1}^{2n-2}(-1)^{k}(2\lambda)^{-k}+(1-\lambda^2)M_\lambda+\tfrac{1}{6}-\tfrac{\lambda}{3}\biggr|\lambda^{2n-2}\\
&\geq(-\tfrac{1}{12\lambda}+\tfrac{1+\lambda}{3}+\tfrac{1}{6}-\tfrac{\lambda}{3})\lambda^{2n-2}.
\end{align*}
Simple calculations show that in both of the above inequalities, the right-hand sides for $n=1$ are uniformly positive for $\tfrac12 < \lambda < 1$. Hence, the claim follows.}
\end{proof}

\begin{lemma}\label{lemma:crux}
    Let $x\in\R^2$ be such that $\Xl\cap (V_0+x)=\emptyset$. Then there are constants $C,K>0$ depending only on $\lambda$ such that
    \begin{equation*}
        \#\Sigma_{n}([V_0+x]_{C\lambda^{n}})\leq K,
    \end{equation*}
    for all $n\in\N$.
\end{lemma}

\begin{proof}
It suffices to find a constant $C>0$ such that
\begin{equation}\label{eq:reduction}
    \#\Sigma_{2n}^1([V_0+x]_{C\lambda^{2n}})\leq K,
\end{equation}
for all $n\in\N$, since then, by symmetry, for any $n\in\N$ we have $\#\Sigma_{n}([V_0+x]_{C\lambda^{n}})\leq 4K$. Let us first assume that $\Xl\subset H^+(x,V_0)$, that is, the line $V_0+x$ lies below $\Xl$. By induction, it is easy to see that for any $C\leq \frac{1}{2\lambda}$ the set $\Sigma_{2n}^1([V_0+x]_{C\lambda^{2n}})$ contains at most the word $\iii_1$.

For the other case $\Xl\subset H^-(x,V_0)$ we show by induction that for every $n\in\N$ there is a constant $C>0$ such that the set $\Sigma_{2n}^1([V_0+x]_{C\lambda^{2n}})$ contains at most the word $\iii_{\max}|_{2n}$. By Lemma \ref{lemma:vertical-distance}, we may choose a constant $C>0$ such that for any $n>1$ and $\iii\in\{\iii_n^L,\iii_n^R,\iii_n^{RR}\}$, we have $\dist_y(\overline{x}_{\max},\varphi_{\iii}(\Xl))\geq C\lambda^{2n}$. Moreover, if $n=1$, then we have $\dist_y(\overline{x}_{\max},\varphi_{\iii^L_1}(\Xl))\geq C\lambda^{2}$ and therefore, $[V_0+\overline{x}_{\max}]_{C\lambda^{2}}=B_y(\overline{x}_{\max},C\lambda^2)$ does not intersect the set $\varphi_{\iii_1^L}(\Xl)$. Since the only words in $\Sigma^1_2$ are $\iii^L_1$ and $\iii_{\max}|_2$, we see that $\Sigma_{2}^1([V_0+\overline{x}_{\max}]_{C\lambda^{2}})$ contains at most the word $\iii_{\max}|_{2}$. Since $\Xl$ lies below $V_0+\overline{x}_{\max}$, this is also true for any $x\in\R^2$ satisfying $\proj_{V_{\infty}}(x)>\proj_{V_{\infty}}(\overline{x}_{\max})$.

Now assume that the set $\Sigma_{2(n-1)}^1([V_0+x]_{C\lambda^{2(n-1)}})$ contains at most the word $\iii_{\max}|_{2(n-1)}$. Since $C\lambda^{2n}<C\lambda^{2(n-1)}$, the only cylinders that could intersect $[V_0+x]_{C\lambda^{2n}}$ are the ones corresponding to the children of $\iii_{\max}|_{2(n-1)}$, which are precisely the cylinders determined by $\iii_{\max}|_{2n}$, $\iii_n^L$, $\iii_n^R$, and $\iii_n^{RR}$. By relying on Lemma \ref{lemma:vertical-distance}, we see that $[V_0+x]_{C\lambda^{2n}}=B_y(x,C\lambda^{2n})$ does not intersect $\varphi_{\iii_n^L}(\Xl)$, $\varphi_{\iii_n^R}(\Xl)$ or $\varphi_{\iii_n^{RR}}(\Xl)$, which finishes the proof.
\end{proof}

\begin{proof}[Proof of Lemma \ref{lemma:cylinder-bound}]
Let $\Theta\geq K_{\lambda}$, $V_t\in \RP_{\Theta}$, and $x\in\R^2$ be such that $\Xl\cap (V+x)=\emptyset$. The case $t=0$ follows from Lemma \ref{lemma:crux}, so by symmetry we may assume that $t>0$. We first consider the case $\Xl\subset H^+(x,V_0)$. Without loss of generality, we may assume that the first coordinate of $x$ is $1$ and note that then $\Xl\subset H^+(x,V_0)\cap H^-(x,V_{\infty})$. By Lemma \ref{lemma:crux}, there are constants $c_1,K>0$ such that $\#\Sigma_{n}([V_0+x]_{c_1\lambda^{n}})\leq K$ for all $n\in\N$ and therefore, by Lemma \ref{lemma:geometric1}, there is a constant $c_2>0$ such that $\#\Sigma_{n}([V_t+x]_{c_1c_2\lambda^{n}})\leq K$ for all $n\in\N$.

For the case $\Xl\subset H^-(x,V_0)$, we let $t_0=0$ and for any $n\in\N$ we define $t_n=(2\lambda)^n\sum_{k=1}^n(2\lambda)^{-k}$. Further, let $\delta_0=0$ and
\begin{equation*}
    \delta_n=\frac{t_{n-1}+t_n}{2}.
\end{equation*}
Clearly $\delta_n$ is strictly increasing with $n$. Let $k\geq1$ be the unique integer satisfying $\delta_{k-1}\leq t < \delta_k$. Since $t\leq \Theta$, there is a natural number $N=N(\Theta,\lambda)$ such that $k\leq N$. By a geometric argument similar to the proof of Lemma \ref{lemma:crux}, it is possible to show that there is a constant $C=C(\Theta)>0$ such that $\Xl\cap [V_t+x]_{C}=\varphi_{\iii_1|_k}(\Xl)\cap [V_t+x]_{C}$. Therefore, by Lemma \ref{lemma:pull-back-cylinder}, for any $n>k$ and $r<C$ we have 
\begin{align}\label{eq:pull-back-bound}
    \Sigma_n([V+x]_{r})&=\{\iii_1|_k\jjj\in\Sigma_{n}\colon \varphi_{\jjj}(\Xl)\cap (\varphi_{\iii_1|_k}^{-1}([V_t+x]_{r}))\ne\emptyset\}\\
    &\subset\{\iii_1|_k\jjj\in\Sigma_n\colon \jjj\in\Sigma_{n-k}([V_{t_{\iii_1|_k}}+y]_{c\lambda^{-k}r})\}.\nonumber
\end{align}
By Lemma \ref{lemma:pull-back-cylinder} and the choice of $k$, we have
\[
t_{\iii_1|k}=-\sum_{n=1}^{k}(2\lambda)^{-n}+(2\lambda)^{-k}t
\]
and since $\delta_{k-1}\leq t<\delta_k$ we get
\begin{equation*}
    -\tfrac{1}{2}(2\lambda)^{-k}\leq t_{\iii_1|_k}< \tfrac{1}{2}(2\lambda)^{-(k+1)},
\end{equation*}
and so, we see that $|t_{\iii_1|_k}|\leq 1-\varepsilon$, where $\varepsilon=(1-\frac{1}{4\lambda})$. By symmetry, we may assume without loss of generality that $0<t_{\iii_1|_k}\leq 1-\varepsilon$ and that $\proj_{V_0}(y)=\frac{1}{3}$. Let $c,K>0$ be as in Lemma \ref{lemma:crux}. Then for any $n\geq N$, we have $\#\Sigma_{n-k}([V_0+y]_{c\lambda^{n-k}})\leq K$ and, by applying Lemmas \ref{lemma:empty-cone} and \ref{lemma:geometric2}, we see that there exists a constant $c>0$ such that
\begin{equation*}
    \#\Sigma_{n-k}([V_{t_{\iii_1|_k}}+y]_{c\lambda^{n-k}})\leq K.
\end{equation*}
In particular, using (\ref{eq:pull-back-bound}), we have
\begin{equation*}
    \#\Sigma_{n}([V+x]_{c\lambda^{n}})\leq\#\Sigma_{n-k}([V_{t_{\iii_1|_k}}+y]_{c\lambda^{n-k}})\leq K,
\end{equation*}
for all $n$ large enough such that $c\lambda^{n-N}< C$, and the claim follows.
\end{proof}


\end{document}